\documentclass{amsart}
\usepackage[utf8]{inputenc}
\usepackage{amssymb}
\usepackage{amsmath}
\usepackage{array}
\usepackage{hyperref}
\usepackage{a4wide}
\usepackage{xypic}
\usepackage{multirow}
\usepackage{xcolor}

\usepackage{tikz}
\usepackage{tikz-cd}
\usetikzlibrary{shapes.geometric, arrows}

\newcommand{\Lg}{\mbox{$\mathfrak g$}}
\newcommand{\Lk}{\mbox{$\mathfrak k$}}
\renewcommand{\Lp}{\mbox{$\mathfrak p$}}

\newcommand{\La}{\mbox{$\mathfrak a$}}

\newcommand{\Lh}{\mbox{$\mathfrak h$}}

\newcommand{\Lt}{\mbox{$\mathfrak t$}}
\newcommand{\Ll}{\mbox{$\mathfrak l$}}

\newcommand\fieldsetc{\mathbb}

\newcommand{\Z}{\fieldsetc{Z}}
\newcommand{\R}{\fieldsetc{R}}
\newcommand{\C}{\fieldsetc{C}}
\newcommand{\Q}{\fieldsetc{H}}

\newcommand{\B}{\mathcal{B}}

\newcommand{\ad}{\mathrm{ad}}
\newcommand{\Ad}{\mathrm{Ad}}

\newtheorem{thm}{Theorem}[section]
\newtheorem{cor}[thm]{Corollary}
\newtheorem{prop}[thm]{Proposition}
\newtheorem{lem}[thm]{Lemma}

\theoremstyle{remark}
\newtheorem{rmk}[thm]{Remark}
\newtheorem{eg}[thm]{Example}

\newcommand{\Pf}{{\em Proof}. }
\newcommand{\EPf}{\hfill$\square$}

\title{Almost symmetric submanifolds}

\author{Claudio Gorodski and Carlos Olmos}\thanks{The first named author
acknowledges partial support from FAPESP Thematic Grant
22/16097-2 and CNPq Fellowship 304252/2021-2, Brazil. The second named
author acknowledges partial support by FaMAF,
Universidad Nacional de Córdoba and CIEM-CONICET, Argentina.}



\begin{document}

\maketitle

\begin{abstract}
We introduce the class of almost symmetric submanifolds
of Euclidean space, a close relative of symmetric submanifolds and
(contact) sub-Riemannian symmetric spaces. More specifically:
\begin{itemize}
\item Homogeneous classification: We prove that
  every full irreducible almost symmetric
  submanifold of Euclidean space is either: a most singular orbit of
  an s-representation; or an almost singular orbit, which can be realized as
  a holonomy tube over a symmetric submanifold;
  or a codimension $3$ submanifold. We include tables of all
  examples with Lie-theoretic data. 
\item Inhomogeneous structure: We prove that any inhomogeneous almost symmetric
  submanifold has cohomogeneity one and describe possible structures, including
  multiply-warped products. 
\item Connection with sub-Riemannian geometry: We interpret almost symmetric submanifolds as embeddings of sub-Riemannian symmetric spaces, highlighting the interplay between extrinsic and intrinsic symmetry.
  \item New invariant: We propose the co-index of extrinsic symmetry as a potential tool to study and hierarchize highly symmetric submanifolds.
\end{itemize}
\end{abstract}

\section{Introduction}

A submanifold of a Euclidean space is called \emph{(extrinsically) symmetric}
if it is reflectionally symmetric with respect to any affine normal
space. Symmetric submanifolds were introduced and classified
by Ferus~\cite{Fe}, who locally characterized them by the
condition that the second fundamental form be parallel (see also~\cite{St,ET}).
In particular Ferus showed that symmetric submanifolds are homogeneous and
arise as special orbits of the isotropy representation of a symmetric space,
or \emph{s-representation}, for short. Symmetric submanifolds
split extrinsically, and the irreducible ones are classified
by the work of Kobayashi and Nagano~\cite{KN}. Symmetric submanifolds enjoy a rich interplay between extrinsic geometry and representation theory, as the geometry is closely tied to the structure of the isotropy representation.

The goal of this paper is to introduce a related class of Euclidean
submanifolds and obtain similar structure and classification results.
An \emph{(extrinsic) almost symmetry} of a Euclidean submanifold
$M$ at a point $p$
in $M$ is an involutive ambient isometry that preserves $M$,
fixes the affine normal space at~$p$ pointwise, and
whose fixed point set in the tangent
space has dimension one.
A submanifold will be called \emph{(extrinsically)
almost symmetric} if it admits an almost symmetry at every point.
This notion interpolates between fully symmetric submanifolds
and certain contact sub-Riemannian symmetric spaces (see below). 

Almost symmetric submanifolds are interesting because, despite relaxing full symmetry, they retain strong rigidity and homogeneity properties in many cases, and in general have at most cohomogeneity one. In particular, any homogeneous almost symmetric submanifold can be viewed as a compatible embedding of a sub-Riemannian symmetric space into Euclidean space, where the restriction of the almost symmetry gives a sub-symmetry in the sense of Strichartz~\cite{St}. Sub-Riemannian symmetric spaces have been classified in low dimensions~\cite{St,FG2,Al} and for contact structures~\cite{FG,BFG}. Our work shows that most homogeneous examples of almost symmetric submanifolds arise from embeddings of contact sub-Riemannian symmetric spaces.

Recall that a \emph{sub-Riemannian manifold} is a smooth manifold $M$ carrying
a distribution $\mathcal D$ (a subbundle of the tangent
bundle $TM$)
and a smooth Riemannian metric $g$ defined only on vectors in~$\mathcal D$.
A sub-Riemannian symmetric space is a Sub-Riemannian manifold
that admits an involutive isometry at each point, called a sub-symmetry, that preserves the $\mathcal D$ and whose differential is minus the identity on $\mathcal{D}$~\cite{St}. When $\mathcal{D}$ is contact, homogeneity follows from the existence of sub-symmetries~\cite{FG}. Homogeneous almost symmetric submanifolds provide concrete realizations of certain contact sub-Riemannian symmetric
spaces inside Euclidean space. 

More generally, one can define the
\emph{co-index of extrinsic symmetry} of $M$ at $p$ as the smallest
integer $k$ such that there is an involutive ambient isometry preserving
$M$, fixing the affine normal space
at $p$ pointwise, and whose fixed point set in the tangent
space has dimension~$k$. An interesting problem is to hierarchize
submanifolds in terms of this invariant.

The first author wishes to thank Wolfgang Ziller for useful discussions.

\section{Main results}

Next we state the main results in this paper. We first
enunciate structure and classification theorems in the homogeneous case. Throughout the paper homogeneous submanifold always stands for 
extrinsically homogeneous submanifold
(see section~\ref{prelim} for unexplained notation and terminology). 

\begin{thm}\label{thm:struct}
  Let $M$ be a full irreducible compact homogeneous 
submanifold of Euclidean space of codimension at least two. 
Assume that $M$ is an almost symmetric submanifold. 
Then:
\begin{enumerate}
    \item If the codimension of $M$ is greater that three, then $M$ is an orbit of an irreducible $s$-representation.
    \item If $M$ is an orbit of an $s$-representation, then $M$ is either
      a most singular orbit or an almost most singular orbit. 
    \item If $M$ is an almost most singular orbit of an $s$-representation, then it is a partial holonomy tube $(\bar M)_\xi$ of an symmetric submanifold $\bar M$. Moreover, the normal vector $\xi$ belongs to a  $2$-dimensional irreducible factor of the normal holonomy of $\bar M$. Conversely, any holonomy tube over
      a symmetric submanifold through a sufficiently small vector in a two dimensional factor of the normal holonomy is an almost symmetric
      submanifold.
    \item If $M$ is not an almost most singular orbit of an $s$-representation, then there  exists a unique almost symmetry $\sigma _p$ at $p$, for all $p\in M$. In particular $\sigma _p$ is invariant under the isotropy group at~$p$. 
      Further, there are examples where the almost symmetry is not unique.
\end{enumerate}
\end{thm}

\begin{thm}\label{thm:classif}
\begin{itemize}
\item[(a)] Every full irreducible almost symmetric orbit $M=Ka$ of the
s-representation associated to a symmetric pair $(\Lg,\Lk)$ is 
listed in the following table. Each $M$ is a circle bundle over $M/S^1$. 
{\footnotesize
\[ \begin{array}{cccccc}
Nr. & M/S^1 & \Lg & \Lk & \Lk_a & \mbox{Conditions}\\
\hline
1&\C P^{q-1} & \mathfrak{sp}(q) & \mathfrak u(q) & \mathfrak{u}(q-1) & q\geq2\\
2&S^2\times S^2 & \mathfrak g_2 & \mathfrak{so}(4) & \mathfrak{so}(2) \\
3&Sp(3)/U(3) \times S^2& \mathfrak{f}_4 & \mathfrak{sp}(3)\oplus\mathfrak{sp}(1) & \mathfrak{su}(3) \oplus\mathfrak{u}(1)\\
4&Sp(4)/U(4) & \mathfrak{e}_6 & \mathfrak{sp}(4) & \mathfrak{su}(4) \\
5&SU(8)/S(U(4)\times U(4)) & \mathfrak{e}_7 & \mathfrak{su}(8) & \mathfrak{su}(4)\oplus\mathfrak{su}(4) \\
6&SO(16)/U(8) & \mathfrak{e}_8 & \mathfrak{so}(16) & \mathfrak{su}(8) \\
7&G_2(\R^{q+1}) & \mathfrak{su}(q+1) & \mathfrak{so}(q+1) & \mathfrak{so}(q-1) & q\geq2\\
8&S^2 \times S^1  & \mathfrak{so}(5) & \mathfrak{so}(3) \oplus \mathfrak {so}(2) 
& \{0\}\\
9&\C P^{p-1} \times \C P^{q-1} & \mathfrak{su}(p+q)
&\mathfrak{s}(\mathfrak{u}(p)\oplus\mathfrak{u}(q)) 
&\mathfrak{s}(\mathfrak{u}(p-1)\oplus\mathfrak{u}(q-1)\oplus\mathfrak{u}(1)) &p\geq q\geq2\\ 
10&E_6/\{[Spin(10)\times U(1)]/\Z_4\} & \mathfrak{e}_7 & \mathfrak{e}_6\oplus\mathfrak{u}(1) & \mathfrak{so}(10)\oplus\mathfrak{u}(1)\\ 
11&Spin(10)/U(5) & \mathfrak{e}_6 & \mathfrak{so}(10)\oplus\mathfrak{u}(1)&
\mathfrak{su}(5)\oplus\mathfrak{u}(1)\\
12&G_2(\R^p)\times G_2(\R^q) & \mathfrak{so}(p+q) & \mathfrak{so}(p)\oplus\mathfrak{so}(q)
& \mathfrak{so}(p-2)\oplus\mathfrak{so}(q-2)\oplus\mathfrak{so}(2)& p\geq q\geq2\\
13&G_3(\C^6)\times\C P^1 & \mathfrak{e}_6 & \mathfrak{su}(6)\oplus\mathfrak{su}(2) & \mathfrak{s}(\mathfrak{u}(3)\oplus\mathfrak{u}(3)) \\
14&SO(12)/U(6)\times\C P^1 & \mathfrak{e}_7 & \mathfrak{spin}(12)\oplus\mathfrak{su}(2) & \mathfrak{u}(6) \\
15&E_7/\{[E_6\times U(1)]/\Z_3\}\times \C P^1 & \mathfrak{e}_8 & \mathfrak{e}_7\oplus\mathfrak{su}(2) & \mathfrak{e}_6\oplus\mathfrak{u}(1)\\
\end{array} 
 \]
}
\item[(b)]  Every full irreducible compact 
almost symmetric homogeneous submanifold $M$ of Euclidean space which is not 
an orbit of an s-representation is an arbitrary principal orbit of $\rho$, where $\rho$ is one of the 
following three reducible representations of cohomogeneity three:
{\footnotesize
  \[ \begin{array}{cccc}
Nr. & M & \rho & \mbox{Conditions}\\
\hline 
16 & V_2(\R^n) & (SO(n),\R^n\oplus\R^n) & n\geq3 \\
17 & U(2) & (U(2),\C^2\oplus\R^3) & -\\
18 & T^2 \times S^3  &(U(1) \times SU(2) \times U(1),\C^2\oplus\C^2) & - \\
  \end{array} \]
  }
\end{itemize}
\end{thm}

We make a few comments about the above tables.
Examples Nr.~7 and~12 provide different embeddings of the Stiefel manifold
$V_2(\R^n)$ into Euclidean space as an almost symmetric submanifold.

All examples other than Nr.~8 are intrinsically 
contact sub-Riemannian symmetric spaces:
the almost symmetry restricts to a sub-symmetry, and it 
passes to the circle quotient defining a structure of symmetric space
in the quotient; further there is a contact CR structure on $M$ that
defines a K\"ahler structure on the quotient, which then becomes
a compact Hermitian symmetric space. Example Nr.~8 is
intrinsically an odd-contact sub-Riemannian symmetric space of dimension~$4$
(cf.~\cite{FG2}). 

Finally, we consider the inhomogeneous case and prove that 
inhomogeneous almost symmetric
submanifolds must be of cohomogeneity one, 
already in the intrinsic case (see section~\ref{inhomog} for the 
definition). 

\begin{thm}\label{thm:inhomog}
Let $M$ be a complete almost symmetric space, denote by $K$
the closure of the
group of isometries of $M$ generated by the almost symmetries of $M$,
and assume that $K$ is not transitive on $M$. Then:
\begin{enumerate}
\item $K$ acts on $M$ with cohomogeneity $1$.
\item In the open and dense subset $\Omega$ of $M$ consisting of regular 
points of the $K$-action, the almost symmetry at a point is unique. 
\item For a $K^0$-orbit $N$ of codimension $k\geq2$, there is a
  connected normal subgroup $H$ of $K^0$ which acts trivially
  on $N$ and whose isotropy representation at any $p\in N$
  on $T_pM=T_pN\oplus\nu_pN$ is $\mathrm{Id}\times SO(\nu_pN)$.
  In particular, the isotropy algebra $\Lk_p$ contains an ideal
  isomorphic to $\mathfrak{so}_k$. 
\item The connected components of the singular $K$-orbits are 
totally geodesic. 
\item The connected components of the $K$-orbits are (intrinsic)
symmetric spaces.
\end{enumerate}
\end{thm}

\begin{thm}\label{thm:inhomog-classif}
Let $M$ be an inhomogeneous simply-connected complete almost
symmetric space. Then $M$ falls into one
of the following classes:
\begin{enumerate}
\item $M$ is a multiply warped product $\R\times_f N$, where $N$ is a
  simply-connected symmetric space, and each component
  of the warping map $f(t)$, $t\in\R$, scales independently
  each irreducible factor of $N$.
\item $M$ is a Riemannian product $\R^k\times N$,  where $N$ is a
  simply-connected symmetric space, and $\R^k$ carries a 
$SO(k)$-invariant metric. 
\item $M$ is a Riemannian product $S^k\times N$,  
where $N$ is a
  simply-connected symmetric space, and $S^k$ carries 
a $SO(r)\times SO(k-r)$-invariant metric. 
\end{enumerate}
\end{thm}

It is easy to see that the manifolds in Theorem~\ref{thm:inhomog-classif}
can be embedded as into Euclidean space as almost symmetric
submanifolds 
if and only if each irreducible factor
of $N$ is an extrinsic symmetric space.

It is also
clear that the Riemannian universal covering of an almost symmetric space
is as well
almost symmetric, but in principle it could be homogeneous even if the
base manifold is of cohomogeneity one. So the classification of
non-simply connected inhomogeneous complete almost symmetric spaces
is a delicate problem in the intrinsic case. On the other hand,
in the extrinsic case, it is not difficult to obtain a complete
classification of inhomogeneous complete almost symmetric submanifolds
of Euclidean space, but we leave this result to a forthcoming paper so as
not to make this one much longer.

\section{Preliminaries}\label{prelim}

Throughout this article, a homogeneous submanifold stands for an extrinsically homogeneous submanifold.

\subsection{Canonical connections on the normal bundle}\label{P1} 
Let $K$ be a  not necessarily connected closed subgroup of $O(n)$, 
and let $M=K^0\cdot p$, where  $0\neq p\in \mathbb R ^n$, and $(\,)^0$ denotes the identity component. We assume that  $KM=M$. Otherwise, we replace $K$ by the group $\{k\in K: kM=M\}$ which has the same identity component $K^0$. Assume that $M$ is full and irreducible as a submanifold of $\mathbb R^n$;
this is always the case if $K^0$ acts irreducibly (but there exist full
and irreducible orbits of reducible representations).
Further assume that $M$ is not the sphere of radius $\Vert p\Vert$.
Let us consider  an $\mathrm{Ad}(K_v)$-invariant positive definite inner product $b$ on 
${\mathfrak {k}} : = \mathrm{Lie}(K)$, and let 
${\mathfrak{m}}:= {\mathfrak k}^\perp$. Then 
${\mathfrak{m}}$ induces a canonical connection on the principal fiber bundle 
$$0\to  K _p \to  K \to  K/K_p=M\to 0.$$
Observe that neither $K _p$ nor $K$ are  in general connected, but $M$ is so.
Consider the slice representation $\rho$ of $K_p$ on the normal space $\nu _pM$,
and the associated metric vector  bundle $ K\underset{\rho}{\times}\nu _pM$ with the induced canonical connection $\tilde \nabla ^\perp$, which is $K$-invariant. 

We identify $\nu M \, \simeq \, K\times_\rho\nu _vM$ in the natural way: $[g,\xi]\simeq \mathrm dg_p (\xi)$, where $g\in  K$, $\xi \in \nu _pM$. 
Such a canonical normal connection $\tilde \nabla ^\perp$ has the property that given a piecewise differentiable  curve $c: [0,1]\to M$, there exists an element of $g \in  K$, non-unique if $\ker \rho$ is non-trivial,
with $g(c(0)) = c(1)$ and such that $\mathrm dg_{c(0)}$ coincides with 
the $\tilde \nabla ^\perp$-parallel transport $\tilde \tau ^\perp _c$ along $c$. In fact, if $\tilde c$ is a  horizontal lift of $c$ to $K^0$, then 

$$\tilde \tau ^\perp _c: \nu _{c(0)}M \simeq  
[\tilde c(0), \nu _vM] \to \nu _{c(1)}M \simeq  
[\tilde c(1), \nu _vM]$$ is given by 
$$\tilde \tau ^\perp _c([\tilde c (0),\xi])
= [\tilde c (1),\xi] = \tilde c (1)\tilde c (0)^{-1}
[\tilde c (0),\xi] = g[\tilde c (0),\xi].$$
Observe that, in our construction,   $g = \tilde c (1)\tilde c (0)^{-1}$ belongs to $K ^0$.

\begin{rmk}\label{rem:Simons} Assume that $K$ acts on $\mathbb R^n$ as an irreducible $s$-representation, and that the orbit $M = K\cdot p$, $p\neq 0$, is not a sphere. Then  $K = \bar K : = \{g\in SO(n): gM = M \}^0$. In fact, let $R$ the Riemannian curvature tensor associated with the $s$-representation. Since $K\subset \bar K$, then 
  $[\mathbb R^n , R, K]$ and $[\mathbb R^n , R, \hat K]$ are irreducible and non-transitive holonomy systems. Then, by Simons' holonomy theorem \cite{Si},
  $K=\bar K$.
\end{rmk}

We denote by  $\nabla^\perp$ the normal connection on $\nu M$ induced by the 
connection of the ambient Euclidean space, by 
$D^ \perp$ the mixed $K$-invariant tensor 
$$D^\perp:= \nabla ^\perp - \tilde \nabla ^\perp,$$ 
 and by $\mathcal D^K$ the $K$-invariant distribution on $M$ defined by 
 $$\mathcal D_q^K = \{z\in T_qM: D^\perp _z=0\}$$
for $q\in M$.  
 The distribution $\mathcal D^K$ depends on $K$ and on~$b$. In fact, if we replace $K$ by $ K'$, where $K\cdot v = K'\cdot v $, then 
 $\mathcal D^K$ and $\mathcal D^{K'}$ are not in general equal. For instance, the isotropy representation of the oriented Grassmannian 
 $SO(2n)/SO(n)\times SO(n)$ has an extrinsic symmetric orbit $M=SO(n)\times SO(n) \cdot v = SO(n)\times SO(n)/\mathrm{diag}(SO(n))$. In this case the associated canonical normal connection   coincides with $\nabla ^\perp $, and thus  
 $\mathcal D^{SO(n)\times SO(n)} = TM$. But  $M = SO(n) \cdot v$ is also 
a principal orbit of $SO(n)$. In this case the  canonical normal connection associated to $SO(n)$ is flat and
hence different from $\nabla ^\perp$. Thus, 
$\mathcal D^{SO(n)}\neq TM$. 

Let $\nabla ^c$ be the canonical tangent connection on $M$ induced by the canonical connection $\mathcal H$ on the associate bundle $K\times_\chi{T_vM}\simeq TM$, where 
$\chi$ is the isotropy representation of $K _v$. The connection $\nabla ^c\oplus \tilde \nabla ^\perp$, on the bundle 
$TM\oplus \nu M$, has the property that any $K^0$-invariant
 mixed tensor, in particular the second fundamental form $\alpha$, is $\nabla ^c\oplus \tilde \nabla ^\perp$-parallel. Then the  main theorem in \cite{OS}, and Remark \ref{rem:Simons} imply the  following result:

\begin{thm}\label{thm:OS} Let $M=K^0\cdot p$ be an irreducible and full homogeneous submanifold of $\mathbb R^N$, where $K$ is a compact subgroup of $O(N)$, with $KM=M$.  Assume, keeping the  notation of this section,  that  $\mathcal D^K= TM$ (equivalently, $\tilde \nabla ^\perp = \nabla ^\perp $). Then there exists  a compact  subgroup $\hat K \subset SO(N)$, acting as an irreducible $s$-representation, such that    $M=K^0\cdot p= \hat K \cdot p$. Moreover,  $K^0\subset \hat K$ and $\mathcal D^{\hat K}= TM$.
\end{thm}

\subsection{Rank rigidity and holonomy tubes}
The \textit{rank} of a homogeneous submanifold $M$ of Euclidean space is
the maximal number of linearly independent parallel normal vector fields.
Equivalently, the rank is the dimension of the flat factor of the
restricted normal holonomy group. If $M$ is contained in a sphere, then
the position vector provides a umbilical parallel normal vector field. 

The following result can be found in \cite[Corollary 5.1.8]{BCO}:

\begin{thm}[\cite{O1,O2}] \label{thm:rank2} Let $M^n$, $n\geq 2$, be a homogeneous full irreducible submanifold of $\mathbb R^ N$ with 
$\mathrm{rank}(M)\geq 1$. Then $M$ is contained
in a sphere. In addition, if $\mathrm{rank}(M)\geq 2$, then $M$ is an orbit of an $s$-representation.
\end{thm}

Next recall the following theorem (it holds without the
compactness assumption).

\begin{thm}[\cite{O2}, Theorem~C] \label{thm:O1C}
  Let $M^ n=K\cdot p$, $n\geq 2$,  be a full irreducible
  compact homogeneous submanifold of $\mathbb R^N$ and let $k\in K$, $q\in M$. Then there exists a piecewise differentiable curve 
  $c:[0,1]\to M$ 
  with $c(0)=p$, $c(1)=kp$ and
  such that $\mathrm dk\vert_{\nu _pM} = \tau _c^{\perp}$,
    where $\tau ^\perp$ denotes the $\nabla ^{\perp}$-parallel transport.
\end{thm}

We now turn to  the definition of a holonomy tube
(see \cite[sec.~3.4.3]{BCO}). If $N$ is an Euclidean submanifold, $p\in N$ and
and $\eta_p\in \nu _pM$, the holonomy tube $(N)_{\eta _p}$ at $p$ is defined by 
$$(N)_{\eta _p}: = \{\gamma (1) + \tau _\gamma ^\perp(\eta _p)\},$$
where $\gamma:[0,1]\to N$ runs through all
piecewise differentiable curves with $\gamma (0)=p$.
If $\Vert\eta _p\Vert$ is less than the focal distance, then $(N)_{\eta _p}$ is a submanifold and $\dim (N)_{\eta _p} = \dim N 
+ \dim (\Phi (p)\cdot \eta _p)$. If the connected  component $\Phi^* (p)$ of $\Phi (p)$, the so-called restricted normal holonomy group, does not fix $\eta _p$, then  $N = ((N)_{\eta _p})_{\bar \eta} $ is a  parallel focal manifold of $(N)_{\eta _p}$, where $\bar \eta$ is the non-umbilical parallel normal vector
field of $(N)_{\eta _p}$ with $\bar \eta (p+\eta _p) = -\eta _p$.

Observe that Theorem \ref{thm:O1C} implies that  $\rho (K_p)\subset \Phi (p)$, where $\rho$
is the slice representation and $\Phi (p)$ is the normal holonomy group of $M$ at $p$. Therefore:

\begin{cor}\label{cor:O1C}
  Let $M^n=K\cdot p$, $n\geq 2$,  be a full irreducible compact homogeneous
  submanifold of $\mathbb R^N$ and let $\eta _p\in\nu _pM$. Then  $K\cdot(M)_{\eta _p}= (M)_{\eta _p}$.
\end{cor}

\begin{cor}\label{cor:holhom} Let $M^n =K\cdot p$, $n\geq 2$,  be a
  full irreducible compact homogeneous submanifold of $\mathbb R^N$ and let
  $\eta _p\in\nu _pM$ be sufficiently small so that the parallel
  submanifold $(M)_{\eta _p}$ is an embedded submanifold. Assume that $K$ acts transitively on $(M)_{\eta _p}$ and that $\eta _p$ is not fixed by the restricted normal holonomy group $\Phi^*(p)$. Then $M$ is an orbit of an $s$-representation.  
\end{cor}

\Pf By assumption $K$ acts transitively on $(M)_{\eta_p}$. This implies
that the  non-umbilical parallel normal vector field $\bar \eta$
with $\bar\eta(p+\eta_p)=-\eta_p$ yields $((M)_{\eta _p})_{\bar \eta}=M$.
$(M)_{\eta_p}$ is a full irreducible compact homogeneous  Euclidean submanifold
of dimension and rank at least $2$. Due to Theorem~\ref{thm:rank2},
$\hat K\cdot (p+\eta _p) = (M)_{\eta _p}$,
where $\hat K$ acts as an $s$-representation.
Finally, owing to Theorem~\ref{thm:O1C}, $\bar \eta$ is  $\hat K$-invariant,
so $M=\hat K\cdot p$, and this is also an orbit of an $s$-representation.
\EPf

\begin{rmk}\label{rem:equiv}
    From Corollary \ref{cor:O1C} it follows that the assumption of Corollary \ref{cor:holhom} can be replaced  by:
    $(K_p)^0\cdot \eta _p = \Phi^*(p)\cdot \eta _p$ and this orbit is
    not a point.
\end{rmk}

We finish this section with a result that will be used later.

\begin{thm}\label{thm:cod1}
  Let $M^n =K\cdot v$, $n\geq 2$,  be a full irreducible
  compact homogeneous submanifold of 
$\mathbb R^{n+k}$, where $k\geq 4$ . Assume that $\rho (K_p)$ is a subgroup of codimension at most one of 
$\Phi(p)$.  Then $M$ is an orbit of an $s$-representation.     
\end{thm}
\begin{proof}
  If the codimension is zero, then $\rho((K_p)^0)=\Phi^*(p)$.
Either we can find a non-trivial orbit of $\Phi^*(p)$ in $\nu_pM$   
and we apply Corollary~\ref{cor:holhom} and Remark~\ref{rem:equiv}),
or there are no non-trivial orbits and we apply Theorem~\ref{thm:rank2};
in both cases we deduce that $M$ is an orbit of an $s$-representation.

So let us assume that the codimension is one. 
The compactness of $\Phi^*(p)$ implies that 
$\rho((K_p)^0)$ is a normal subgroup.
In view of the normal holonomy theorem~\cite{BCO},
$\Phi^*(p)$ acts on $\nu _pM$ 
as an $s$-representation, up to a component $V_0$ of fixed vectors.
The case  $\dim V _0\geq 2$ is again covered by Theorem~\ref{thm:rank2},
so we may assume that  $V_0 = \mathbb R p$.
Consider $\nu _pM =\mathbb R p\oplus V_1\oplus\cdots \oplus V_r$,
$\Phi^*(p)$-irreducible decomposition, where
$\Phi^* (p)= \Phi _1 \times \cdots \times \Phi_r$,
and $\Phi _i$ acts irreducibly on $V_i$ and trivially on $V_j$, for~$i\neq j$. Denote by 
$\rho_i((K_p)^0))\subset \Phi _i$ the projection of $\rho((K_p)^0))$ to
$V _i$ for $i=1,\ldots, r$. Observe that $\rho((K_p)^0)\subset
\rho_1((K_p)^0)\times \cdots \times \rho_r((K_p)^0)$.

Suppose first that $r\geq 2$. 
Since the codimension of $\rho((K_p)^0)$~in $\Phi^*(p)$ is one,
there must exist an index $i$,
which we may assume to be $1$, such that $\rho_1((K_p)^0)=\Phi_1$.
Let $\eta_p \in V_1$ be a sufficiently small nonzero vector
such that $(M)_{\eta _p}$ is an embedded submanifold.
Then the assumptions of Corollary \ref{cor:holhom} are satisfied
so that~$M$ is an orbit of an $s$-representation
(compare Remark \ref{rem:equiv}).

It remains to analyse the case in which $r=1$. Since $k\geq 4$,
$\dim V_1\geq 3$.
Now $\rho_1((K_p)^0)$ has codimension $1$ in $\Phi_1$.
Consider the case in which $\Phi _1$
acts transitively on the sphere of $V_1$ of radius
$\Vert \eta _p\Vert$.
If $\rho _1((K_p)^0)$ also acts transitively on this sphere,
then $K$ acts transitively on the holonomy tube $(M)_{\eta _p}$,
where $0 \neq \eta _p \in V _1$ is sufficiently small.
Then  $M$ is an orbit of an $s$-representation by Corollary~\ref{cor:holhom}.
If $\rho _1((K_p)^0)$ acts with cohomogeneity~$1$ on the sphere,
then all orbits must have codimension~$1$, since $\rho _1((K_p)^0))$ is a normal subgroup of $\Phi _1$. This is a contradiction because
an action of cohomogeneity~$1$ on a sphere is polar and has
thus singular orbits.
Finally we  assume that  $\Phi_1$ does not act transitively on the sphere.
We may choose $\eta_p$ such that its holonomy orbit is principal.
By Corollary~\ref{cor:holhom} we may assume that
$\rho _1((K_p)^0)$ is not transitive on the  principal orbit
$\Phi_1\cdot \eta _p$. Then, again by normality of
$\rho _1((K_p)^0)$ in~$\Phi _i$, all orbits of $\rho_1((K_p)^0)$
in the irreducible isoparametric submanifold $\Phi_1\cdot \eta _p \subset V_1$ have codimension $1$. There must exist a curvature distribution
$E$ of~$\Phi_1\cdot\eta_p$
such that $E_{\eta_p} + T_{\eta_p}(\rho_1((K_p)^0)\cdot\eta_p)=
T_{\eta_p}(\Phi_1\cdot\eta_p)$. If we pass to a singular orbit
$\Phi_1\cdot\bar\eta$ that focalizes the eigendistribution $E$, then 
$\rho_1((K_p)^0)\cdot\bar\eta=\Phi_1\cdot\bar\eta$.
Due to Remark~\ref{rem:equiv}, $M$ is an orbit of an $s$-representation.
This finishes the proof. 
\end{proof}

\section{Involutions of orbits of s-representations}\label{sec:invol}

There are certain canonical extrinsic isometries of
order~$2$ of orbits of s-representations, which were
constructed in~\cite{BOR}. Herein we recall their properties
and present an alternative construction.

\begin{prop}
Let $X=G/K$ be a simply-connected
symmetric space of compact type without Euclidean factor,
where $G$ is simply-connected and $K$ is connected, 
write $\Lg=\Lk+\Lp$ for the decomposition of the Lie algebra
of $G$ into the eigenspaces of the involution, and
consider the orbit $M=Ka_0$   
of the isotropy representation of $K$ through a point $a_0\in\Lp$. 
Then there exists a canonically defined
involutive isometry $f$ of $X$ such that:
\begin{enumerate}
\item $f$ fixes the basepoint $o$ of $X$. 
\item $df_{o}$ preserves all $K$-orbits in $T_{o}X\cong\Lp$, fixes $a_0$ and restricts to the identity along the normal space $\nu_{a_0}M$.
\item In case $X$ is irreducible and of rank bigger than one,
$M$ is extrinsically symmetric if and only if the fixed point set
of $f$ is precisely $\nu_{a_0}M$. 
\end{enumerate}
\end{prop}

\Pf Let $\La$ be a Cartan subspace of $\Lp$ containing $a_0$
and consider the restricted root space decomposition
\[ \Lk=\Lk_0+\sum_{\alpha\in\Delta^+}\Lk_\alpha,\quad
\Lp=\La+ \sum_{\alpha\in\Delta^+}\Lp_\alpha. \]
Recall that the restricted roots are purely imaginary valued
on $\La$ since $X$ is of compact type.

Denote by $\alpha_1,\ldots,\alpha_q$ the simple
restricted roots and note that
\begin{equation}\label{tang}
  T_a(Ka) = \sum_{\alpha:\alpha(a)\neq0}\Lp_\alpha
\end{equation}
and
\begin{equation}\label{norm}
  \nu_a(Ka) = \La+ \sum_{\alpha:\alpha(a)=0}\Lp_\alpha
\end{equation}
for all $a\in\La$. Therefore, by replacing $a_0$ with another
element of~$\La$ whose $K$-orbit has the same tangent space,
we may assume that for each~$i$, $\alpha_i(a_0)=0$ equals
either~$0$ or~$\sqrt{-1}$. 
If we extend~$\La$ to a Cartan subalgebra $\Lh$ of $\Lg$,
it follows that all roots of $\Lg$ relative to $\Lh$
take integer values on $\frac1{\sqrt{-1}}a_0$ and thus $\exp 2\pi a_0$ lies in the
center~$Z(G)$ of~$G$.
Hence
\begin{equation}\label{inv}
  e^{2\pi\mathrm{ad}_{a_0}}=\mathrm{Ad}_{\exp2\pi a_0}=1.
\end{equation}

For each $x\in\Lk_\alpha$
we denote by $\hat x\in\Lp_\alpha$ the vector related to $x$
in the sense that
\[ \ad_a x =-\sqrt{-1}\alpha(a)\hat x,\quad \ad_a\hat x=\sqrt{-1}\alpha(a) x \]
for all $a\in\La$. 
Similarly, each $y\in\Lp_\alpha$
we denote by $\check y\in\Lk_\alpha$ the vector related to $y$
in the sense that
\[ \ad_a y =\sqrt{-1}\alpha(a)\check y,\quad \ad_a\check y=-\sqrt{-1}\alpha(a) y \]
for all $a\in\La$. 
An easy calculation yields that
\begin{equation}\label{exp-k}
 e^{t \mathrm{ad}_a}\cdot x =\cos(t|\alpha(a)|)x +\sin(t|\alpha(a)|)\hat x 
\end{equation}
for $x\in\Lk_\alpha$, and
\begin{equation}\label{exp-p} e^{t \mathrm{ad}_a}\cdot y =
  \cos (t |\alpha(a)|)y -\sin (t|\alpha(a)|)\check y 
\end{equation}
for $y\in\Lp_\alpha$. 

Consider the automorphism 
\[ \Phi=e^{\pi\mathrm{ad}_{a_0}}=\mathrm{Ad}_{\exp\pi a_0} \]
of $\Lg$. It follows from~(\ref{exp-k}) and ~(\ref{exp-p})
that $\Phi$ preserves $\Lk$ and $\Lp$. Also, $\Phi:\Lp\to\Lp$
is an isometry since the inner product is $\mathrm{Ad}$-invariant. 
Since $G$ is simply-connected, there is a
unique  automorphism $\varphi$ of $G$ such that $d\varphi=\Phi$. 
Define $f:X\to X$ to be $f(gK)=\varphi(g)K$. Then $f(o)=f(1K)=1K=o$ and 
$df_{o}=\Phi$ is a linear isometry. 
Since~$f(hgK)=\varphi(h)\varphi(g)K$
for $g$, $h\in G$, we now get that $f$ is a global isometry. 
It also follows from~(\ref{inv}) that 
$\Phi$ and $f$ are involutions.

Owing to the fact that $\La$ is Abelian and from~(\ref{norm}), we see
that $\Phi$ restricts to the identity along~$\nu_{a_0}M$. Further,
\begin{align*}
\Phi(\mathrm{Ad}_kb)&=\Phi\left(\frac{d}{dt}\Big|_{t=0}k\exp tbk^{-1}\right)\\
&=\frac{d}{dt}\Big|_{t=0}\varphi(k)\exp t\Phi(b)\varphi(k)^{-1}\\
&=\mathrm{Ad}_{\varphi(k)}\Phi(b)
\end{align*}
for all $k\in K$, $b\in \La$, so $\Phi$ preserves all $K$-orbits in~$\Lp$. 

Note that on $\Lp_\alpha\subset T_{a_0}M$ we have that
\begin{equation}\label{Phi-on-palpha}
  \Phi|_{\mathfrak p_\alpha}=(-1)^{|\alpha(a_0)|}\mathrm{id},
\end{equation}
where $0\neq|\alpha(a_0)|\in\Z$. 

Finally, $M=Ka_0$ is extrinsically symmetric if and only if $a_0$ can be taken  
dual to a simple restricted root $\alpha_i$ with coefficient~$1$ in the 
highest restricted root~\cite[Remarks, p.523]{EH}.
This implies that $\alpha(a_0)=1$ 
for every positive restricted root $\alpha$ that does not kill~$a_0$.  
Hence $\Phi|_{T_{a_0}M}=-\mathrm{id}$. \EPf

\begin{rmk}

Put $g=\exp \pi a_0$. Then $\sigma(g)=g^{-1}=g$, so
$\exp\pi a_0$ normalizes $K$. This implies that the tangent
spaces to $M$ at~$o$ and $\exp\pi a_0\cdot o$ are canonically isomorphic
to~$\Lp$. Now we claim that $df_o=\Phi$ coincides with the 
parallel transport along $\gamma(t)=(\exp ta_0)\cdot o$ from $t=0$ to $t=\pi$.
Recall $X$ is simply-connected and consider the covering $\pi:X\to \bar X$,
where $\bar X$ is the bottom space. It is enough to show that 
$\Phi:\Lp\cong T_{\pi(o)}\bar X\to\Lp\cong T_{\pi(o)}\bar X$ is the 
parallel transport along $\bar\gamma(t)=\pi\circ\gamma(t)$ 
from $t=0$ to $t=\pi$ (note that $\bar\gamma(\pi)=\bar\gamma(0)=\pi(o)$). 
In fact, $g$ satisfies $(dL_g)_o=\mathrm{Ad}_gd(R_{g^{-1}})_o=\mathrm{Ad}_g=\Phi$, since $R_{g^{-1}}=\mathrm{id}$ on $\bar X$, and $t\mapsto\exp ta_0\in\exp[\Lp]$ is a 
one-parameter group of transvections of $\bar X$.
\end{rmk}

\section{The proof of Theorem~\ref{thm:struct}}

Let $M^n=Kp$, $n\geq2$,  be a full irreducible compact homogeneous almost symmetric
submanifold of Euclidean space $V$ of codimension~$k$.
Let $\tilde K$ be the group of ambient isometries that preserve $M$.
Then $\tilde K\supset K$.

\begin{lem}\label{lem:uniq}
If there exist two different almost symmetries at~$p$, then $M$ is an orbit
of an s-representation.
\end{lem}

\Pf 
Denote by~$H$ the kernel of the slice 
representation $\rho: \tilde K_p \to  O(\nu _pM)$. 
We claim that if $H$ does not fix any non-zero vector of $T_pM$, then 
$M$ coincides with an orbit of 
an $s$-representation.
In fact, for arbitrary $\xi$, $\eta \in \nu _pM$, the linear functional on $T_pM$ given by 
$\ell (w) = \langle D_w^\perp\xi ,  \eta\rangle$ is $H$-invariant. 
Our assumption says that the action of $H$ on the dual space $(T_pM)^*$ does not fix any  
non-zero vector. This imples that $\ell = 0$ and proves that $D^\perp = 0$. 
Now we invoke Theorem \ref{thm:OS} to get that $M$ coincides with an orbit of 
an $s$-representation. We finish by noting that
if $\sigma$, $\nu$ are two 
different almost symmetries of $M$ at~$p$, then they belong to $H$
and their common fixed point set in $T_pM$ is zero, so $H$
does not  fix any non-zero vector of $T_pM$, and the argument is complete. 
\EPf

\subsection{Part~(a)}

Suppose, to the contrary, that $M$ is not an s-orbit.
It follows from Lemma~\ref{lem:uniq} that 
the almost symmetry at any $q\in M$ is uniquely defined, 
and we shall denote it by $\sigma_q$. Define
$\mathcal L_q$ to be 
 the $(+1)$-eigenspace of~$\sigma_q$, for all $q\in M$.

\begin{lem}\label{l}
The subspaces $\mathcal L_q$ for $q\in M$ define a $\tilde K$-invariant
autoparallel distribution $\mathcal L$ on $M$.
Its integral manifold thorugh $q\in M$ is the component of the fixed point set 
of $\sigma_q$ containing~$q$. Finally $\mathcal L$ is invariant under all shape 
operators of $M$.
\end{lem}

\Pf By uniqueness of the almost symmetry, $\sigma_q$ commutes with the isotropy action
of $\tilde K_q$, so $\mathcal L$ is $\tilde K$-invariant. A standard argument
shows that the integral manifold $L(q)$ of $\mathcal L$ through $q$ coincides 
with the component of the fixed point set 
of $\sigma_q$ containing~$q$ (cf.~\cite{BCO}, Lemma~9.1.2 and its proof).
Thus $L(q)$ is totally geodesic. Since $\sigma_q$ is a linear isometry of
the ambient Euclidean space $\mathbb R^N$ for all $q\in M$, 
the differentials $\mathrm d(\sigma_q)$ at different points
have constant eigenvalues ($1$ with multiplicity $N-n+1$,
and $-1$ with multiplicity $n-1$). By a continuity argument,  
we must have that $\mathrm d(\sigma_q)_r|_{\nu _rM}= \mathrm{Id}_{\nu _rM}$ for all $r\in L(q)$. 
This implies that $\sigma_q = \sigma_r$ for all $r\in L(q)$. The last assertion
follows from the fact that $d(\sigma_q)_q|_{\nu_qM}=\mathrm{Id}_{\nu _qM}$, so
$d(\sigma_q)_q$ commutes
the shape operators of $M$ at~$q$. \EPf

\medskip

Denote by  $\mathcal L^\perp$ be distribution on $M$ whose fiber at~$q\in M$
is the orthogonal complement of $\mathcal L_q$ in $T_qM$. 
Recall that $\nabla ^\perp$ denotes the normal connection on $\nu M$ 
induced from the Levi-Civit\'a connection and $\tilde \nabla ^\perp$ 
denotes the canonical connection on $\nu M$ induced by $K$. 

\begin{lem}\label{nablas}
$\nabla^\perp=\tilde\nabla^\perp$ along $\mathcal L^\perp$. 
\end{lem}

\Pf Of course the almost symmetries commute with $D^\perp=\nabla^\perp-\tilde\nabla^\perp$.
Now the sign rule implies that $\mathcal D^K=\{v\in TM:D^\perp_v=0\}$ contains $\mathcal L^\perp$
as a subdistribution. Our standing assumption that $M$ is not an s-orbit
precludes the possibility that $\mathcal D^K=TM$ via Theorem~\ref{thm:OS}. 
Since $\mathcal L^\perp$ has codimension one in $M$, we deduce that 
$\mathcal L^\perp=\mathcal D^K$, as wished. \EPf

\medskip

Recall that $\rho:K_p\to O(\nu_pM)$ denotes the slice 
representation of $M$ at~$p$.

\begin{lem}
  There exists a one-parameter subgroup $S$ of $SO(\nu_pM)$
that normalizes $\rho(K_p)$ 
  such that
\[ \rho(K_p) \subset \Phi(p) \subset S\cdot \rho(K_p). \]
\end{lem}  

\Pf Denote by $L_q$ the leaf of $\mathcal L$ through~$q\in M$. 
Since $\mathcal L$ is $K$-invariant, the stabilizer
$K_{L_p}$ acts transitively on $L_p$, so we can write 
$L_p=K_{L_p}/K_p$ as a homogeneous space. In particular
$\dim K_{L_p}=\dim K_p+1$. Owing to the fact that $K_{L_p}$ is a 
compact Lie group, for the Lie algebras this yields
that $\Lk_{L_p}=\Lk_p \oplus \R u$, direct sum of ideals, for some 
$u\in\Lk_{L_p}$.

Let $c:[0,1]\to M$ a smooth loop at $p$. 
By Lemma~\ref{l}, the shape operators of $M$ preserve the 
distribution~$\mathcal L$. Due to the Ricci equation,
$R^\perp(\mathcal L,\mathcal L^\perp)=0$ for the curvature of the normal
connection. Now~\cite[Lem., App.]{O1} can be applied to yield
a factorization of the parallel transport in $\nu M$ along $c$,
\begin{equation}\label{factorization}
  \tau^\perp_c= \tau_\beta^\perp \circ \tau_{\tilde c}^\perp,
  \end{equation}
where $\tilde c$ is a smooth curve defined on $[0,1]$,
everywhere tangent to $\mathcal L^\perp$, originating at~$p$
and ending at a point $q$ in $L_p$, 
and $\beta:[0,1]\to L_p$ is a smooth curve joining $q$ to~$p$. 
Since $L(v)$ is one-dimensional, 
we may assume that $\beta(t)=\exp(-t\lambda u)\cdot q$,
where $\lambda\in\R$ is such that $\exp(-\lambda u)\cdot q=p$. 
The fact that $\tilde c$ is everywhere tangent 
to $\mathcal L^\perp$ implies that $\tau^\perp_{\tilde c}$
coincides with the parallel transport $\tilde\tau^\perp_{\tilde c}$ 
with respect to the 
canonical connection $\tilde\nabla^\perp$, 
and thus $\tau^\perp_{\tilde c}=dk|_{\nu_pM}$ for some $k\in K^0$,
due to Theorem~\ref{thm:O1C}. Set $k'= \exp(-\lambda u)k\in K_p$. 

The action of $u\in \Lk_{L_p}$ on $\xi\in\nu_pM$ given by
\[ \mathcal A_u\xi = \frac{d}{dt}\Big|_{t=0}d(\exp tu)_p\cdot \xi \]
defines a skew-symmetric map on $\nu_pM$ such that
\[ d\exp(-\lambda u)\circ (\tau_\beta^\perp)^{-1} = e^{-\lambda\mathcal A_u} \]
(cf.~\cite[\S~5.2.2]{BCO}).
We substitute into~(\ref{factorization}) to get
\[ \tau_c^\perp = e^{\lambda\mathcal A_u}\circ dk'|_{\nu_pM}. \]
This equation shows that we can take $S$ to be the one-parameter subgroup of $SO(\nu_pM)$
generated by~$\mathcal A_u\in\mathfrak{so}(\nu_pM)$. \EPf

\medskip

Now part~(a) of Theorem~\ref{thm:struct} follows from Theorem~\ref{thm:cod1}. 

\subsection{Part (b)}

Next we assume that $M$ an orbit of an s-representation, but not a most
singular orbit, and we wish to prove that it must be an almost singular orbit.

Herein we work only at the point~$p$. Fix an almost symmetry $\sigma$
of $M$ at~$p$ and denote the subspace of fixed points of $d\sigma_p$
in $T_pM$ by~$L(\sigma_p)$. Since $d\sigma_p$ is the identity on $\nu_pM$,
it commutes with any shape operator at~$p$ and therefore any shape
operator at~$p$ preserves
$L(\sigma_p)$. The restricted normal holonomy $\Phi^*(p)$ coincides with
$\rho((K_p)^0)$, since $M$ is an s-orbit, and the dimension of the
subspace $(\nu_0M)_p$ of fixed points in $\nu_pM$ is the rank~$r$ of $M$.
Note that $r\geq2$ as we are assuming that $M$ is not most singular. 

The family of shape operators $\{A_\xi:\xi\in(\nu_0M)\}$ is commutative,
due to the Ricci equation, and hence simultanelously diagonalizable.
Then there are $g$ parallel normal vector fields $\eta_1,\ldots,\eta_g$,
the so-called curvature normals, with corresponding autoparallel
distributions $E_1,\ldots,E_g$ and an eigenbundle decomposition $TM=E_1\oplus\cdots\oplus E_g$. Note that the curvature normals span $\nu_0M$, for otherwise
$M$ would reduce codimension; hence $g\geq r\geq2$.

We have that $L(\sigma_p)$ is contained in $E_{i_0}(p)$ for some $i_0$,
due to its invariance under the shape operators at~$p$. 
The isometry $\sigma$ maps parallel manifolds to $M$ to parallel
manifolds (including focal manifolds). Since it is the identity on $\nu_pM$
and this subspace meets all $K$-orbits, $\sigma$ preserves all $K$-orbits.
Consider the $K$-orbit $N$ through
$q=p+\xi$, where $\xi$ is a parallel section of $\nu_0M$ satisfying
$\langle\xi,\eta_i\rangle=1$ if and only if $i=i_0$. Then
$N$ is a focal manifold to $M$,
$\sigma(q)=q$ and $d\sigma_q|_{T_qN}$ is the restriction of
$d\sigma_p$ to $T_pM\cap E_{i_0}^\perp$, and hence is $-\mathrm{Id}$. This forces
$d\sigma_q$ to be the identity on the first normal space~$\nu_q^1N$.
Now $\nu^1N$ is $\nabla^\perp$-parallel, since parallel translation is
given by the group, and hence $\nu^1_qN=\nu_qN$, due to the fullness of $N$
(since $K$ acts irreducibly). This shows that $N$ is extrinsically symmetric.
Moreover $L(\sigma_p)=E_{i_0}(p)$ because $d\sigma_q$ does not have $(-1)$-eigenvalues
on $\nu_qN$. Therefore $\dim N=\dim M-1$. Finally, it is known that
extrinsically symmetric submanifolds are most singular orbits of
s-representations, hence $M$ is almost singular. This completes the 
proof of part~(b).

\subsection{Part (c)}

For an almost most singular orbit $M$
of an s-representation which is almost symmetric, let $N$ be
the focal manifold constructed in part~(b). 
Thanks to~\cite[Theorem~4.5.4]{BCO2}, $M$ is a holonomy tube over $N$,
that is, $M=(N)_{\bar\xi}$ for $\bar\xi=-\xi\in\nu_qN$ (notation
as in párt~(b)). 
Since $\dim M=\dim N+1$,
$\xi$ must lie in an irreducible $2$-dimensional component
of the normal holonomy representation of $N$ at~$q$.
Conversely, suppose now $N$ is an extrinsic symmetric submanifold
with symmetry at $q\in N$ given by~$\sigma$.  Take a sufficiently
short non-zero $\xi\in\nu_qN$
and consider the holonomy tube $M=(N)_\xi$. Then $\sigma$ preserves
$M$, fixes $p=q+\xi$, coincides with the identity along $\nu_pM$, and
its set of fixed vectors in $T_pM$ is precisely $T_\xi(\Phi(q)\cdot\xi)$.
We deduce that if $\xi$ lies in a $2$-dimensional irreducible
component of $\Phi^*(q)$ then $M$ is almost symmetric.

\subsection{Part (d)}

For the non-uniqueness statement of part~(d), see subsection~\ref{aso}.

If $M$ is not an orbit of an s-representation, then the uniqueness
result is contained in
Lemma~\ref{lem:uniq}. In view of part~(b) we may thus
assume that $M$ is a most singular
orbit of an s-representation. For each $q\in N$, denote
by $\mathcal F_q$ the family of almost symmetries at $q$. 
Suppose, to the contrary, that $\mathcal F_q$ contains 
at least two elements.
 
For each $\sigma\in\mathcal F_q$, denote by $L(\sigma)$ the one-dimensional
subspace of $T_qM$ of fixed points of $d\sigma_q$. Let also
$V(q)$ denote the linear span in $T_qM$
of $L(\sigma)$ for $\sigma\in\mathcal F_q$. We obtain an
invariant distribution $\mathcal V$ on $M$, which is
moreover invariant under the group $\tilde K$ of extrinsic
isometries of $M$. 

\begin{lem}
$V(q)$ is a proper subspace of $T_qM$ for all $q\in M$. 
\end{lem}

\Pf Since $L(\sigma)$ is invariant under the shape operator
$A_\xi$ for all $\xi\in\nu_qM$, all shape operators
at~$q$ diagonalize simultaneously on $V(q)$. However,
as a most singular orbit of an s-representation, $M$ nas non-flat normal
bundle, so $V(q)$ cannot coincide with $T_qM$. \EPf

\medskip

For $q\in M$, let $H(q)$
denote the closure of the group generated by even products
of elements of $\mathcal F_q$. By or assumption,
$H(q)$ is non-trivial. Let also $W(q)$ denote
the subspace of fixed points of $H(q)$ in $T_qM$.

\begin{lem}\label{w-perp-v}
$W(q)=V(q)^\perp$.   
\end{lem}

\Pf If $v\in V(q)^\perp$ then $v\in L(\sigma)^\perp$ and thus~$d\sigma_q(v)=-v$ 
for all $\sigma\in\mathcal F_q$. In particular $v$ is fixed by~$H(q)$, which
checks the inclusion $V(q)^\perp\subset W(q)$. Let us prove the 
reverse inclusion. By our assumption, there exist two
different elements $\sigma$, $\sigma'$ in~$\mathcal F_q$.
The composition $\sigma\sigma'$ is a rotation in
$L(\sigma)\oplus L(\sigma')\subset V(q)$, whose fixed point set
$(L(\sigma)\oplus L(\sigma'))^\perp$ contains $W(q)$. Since these
$2$-planes generate $V(q)$ for such pairs $\sigma$, $\sigma'$,
this proves that $V(q)^\perp\supset W(q)$, as desired. \EPf

\begin{lem}\label{lem:aux2}
  The subspaces $V(q)$ for $q\in M$
  define a smooth distribution $\mathcal V$ on $M$. 
Further, $\mathcal V$ is contained in the relative
nullity distribution $\mathcal N$
of~$M$ as a submanifold of the sphere of radius $||p||$.
\end{lem} 

\Pf It is obvious that $k\mathcal F_qk^{-1}=\mathcal F_q$
for all $k\in \tilde K_q$. It follows that $H(q)$
is a normal subgroup of $\tilde K_q$ for all $q\in M$.
Therefore the subspaces $W(q)$ give rise to a $\tilde K$-invariant
autoparallel distribution $\mathcal W$ of $M$. The first statement
now follows from $\mathcal V=\mathcal W^\perp$ (Lemma~\ref{w-perp-v}).

Recall that the shape operators of $M$ are simultaneously
diagonalized along $\mathcal V$. This means there exist $\tilde K$-invariant
normal vector fields $\eta_1,\ldots,\eta_g$ on $M$ and a a common
eigenbundle decomposition $\mathcal V=E_1\oplus\cdots\oplus E_g$ such that
$A_\xi|_{E_i}=\langle \xi,\eta_i\rangle\,\mathrm{id}_{E_i}$.
Since $\tilde K^0$ acts as an s-representation, $\tilde K^0$-invariant
normal vector fields must be parallel. This implies that the $\eta_i$
are parallel. However $M$ is a most singular orbit, so $g=1$ and $\eta_1$
is the position vector. Now $A_\xi|_{\mathcal V}=0$ for $\xi$ tangent
to the sphere that contains $M$, that is, $\mathcal V\subset\mathcal N$. \EPf

\medskip

Now we can finish the proof of part~(d). By Lemma \ref{lem:aux2},
the nullity $\mathcal N$ of the second fundamental form of $M$ as a submanifold of the sphere is non-trivial. Since $M$ is homogeneous,
$\mathcal N$ has constant rank. By a remarkable result of
F. Vittone~\cite[Theorem 1]{Vi}, the distribution $\mathcal N^\perp$
is completely non-integrable, that is, any two points in $M$ can be joined
by a piecewise smooth curve tangent to $\mathcal N^\perp$.
On the other hand, the distribution $\mathcal N^\perp$ is contained in the integrable (autoparallel) distribution $\mathcal W =\mathcal V^\perp \subsetneq TM$.
This contradiction shows that there can be
only one almost symmetry at any point of $M$. \EPf

\section{Reduction to automorphisms}\label{sec:reduct}

We now initiate the proof of Theorem~\ref{thm:classif}, namely
the classification of almost symmetric
homogeneous submanifolds of Euclidean space. In this section we
treat the case of orbits of s-representations that are almost symmetric
and relate the almost symmetry to an automorphism of the group
of isometries. 

\begin{lem}\label{1}
  Suppose $X=G/K$ is a simply-connected symmetric space of
  compact type and rank bigger than one, where $G$ is the transvection group.
  Write $\Lg=\Lk+\Lp$ as usual and the denote the basepoint by~$o$. 
  If $\Phi:\Lp\to\Lp$ is an isometry that preserves the orbit
  $M=Ka_0$ for some $a_0\in\Lp$, then $\Phi=df_o$ for a unique
isometry $f$ of $X$ with $f(o)=o$. 
\end{lem}

\Pf We know $K$ is connected.
$\Phi$ maps holonomy tubes of $M$ onto 
holonomy tubes of $M$, and these are the $K$-orbits in~$\Lp$.
It follows that 
$K$ and $\Phi K \Phi^{-1}$ have the same orbits in $\Lp$.
We deduce from this and from~\cite[Remark~8.3.5]{BCO2} (or~\cite[Proposition~4.3.9]{BCO})
that $\Phi$ normalizes $K$. We deduce from~\cite[Lemma~8.3.2]{BCO2}
that $\Phi(R)$ and $R$ are multiples, but they have the same scalar
curvature (since the curvature operators are conjugate, and the scalar
curvature is twice the trace of the curvature operator). It follows
that $\Phi$ preserves $R$. By the Cartan-Ambrose theorem,
$\Phi$ is induced by a global isometry of $X$. Uniqueness is clear since
$f$ is determined by $df_o$. \EPf

\begin{lem}\label{2}
   Suppose $X$ is as in Lemma~\ref{1}.   
   If $\Phi:\Lp\to\Lp$ is an isometry (resp.~an involutive isometry) 
     that preserves
 $M=Ka_0$ for some $a_0\in\Lp$,
 then $\Phi=\tau|_{\mathfrak p}$
 for a unique (resp.~involutive) automorphism of $\tau$ of $\Lg$
that commutes with $\sigma$, that is,
$\tau\in\mathrm{Aut}(\Lg)^\sigma$.
\end{lem}

\Pf Due to Lemma~\ref{1}, we can find an isometry $f$ of $X$ such that
$\Phi=df_o$. In view of~\cite[Proposition~4.1(a), \S4, Ch.~7]{Loos2},
we can write $df_o=\tau|_{\mathfrak p}$
for a unique~$\tau\in\mathrm{Aut}(\Lg)^\sigma$. \EPf

\medskip

Fix a Cartan subspace (that is, a maximal Abelian subspace)~$\La$ of~$\Lp$.
Then we have the (restricted) root space decomposition
of $\Lg$ with respect to $\La$:
\[ \Lg = \Lg_0 + \sum_{\alpha\in\Delta}\Lg_\alpha, \]
where $\Lg_0=\Lk_0+\La$, $\Lk_0$ is the centralizer
$Z_{\mathfrak k}(\La)$ of $\La$ in $\Lk$, and $\Delta$
denotes the root system.

Extend~$\La$ to a $\sigma$-invariant
Cartan subalgebra $\Lh=\Lt+\La$, where $\Lt$ is a CSA of 
$\Lk_0$. We consider the root system of the complexification
$\Lg^c$ with respect to $\Lh^c$; these roots are real-valued
on $\sqrt{-1}\Lt+\La$. Then the elements of $\Delta$ can be seen as
restrictions along $\La$ of roots of $\Lg^c$ with respect to $\Lh^c$. 
If we introduce an ordering of the roots
that takes $\La$ before $\sqrt{-1}\Lt$, as we now do, 
the restriction map $\Lh^c\to\La$ maps positive roots to positive roots,
and simple roots to simple roots. 
Write as usual 
\[ \Lk = \Lk_0+ \sum_{\alpha\in\Delta^+} \Lk_\alpha\quad\mbox{and}\quad 
\Lp = \La+ \sum_{\alpha\in\Delta^+}\Lp_\alpha, \]
with 
\[ \Lk_\alpha+\Lp_\alpha = \Lg_\alpha + \Lg_{-\alpha} \]
for $\alpha\in \Delta^+$.

Recall that the Satake diagram of the symmetric space
associated to a symmetric pair $(G,K)$ is obtained from the
Dynkin diagram of $\Lg$ by painting black the nodes
associated to simple roots of $(\Lg^c,\Lh^c)$ that vanish along $\La$;
and by joining by a curved arrow
two white nodes which are associated to simpĺe roots of $(\Lg^c,\Lh^c)$
that have the same restriction along $\La$. The Satake diagram compĺetely
determines the symmetric space, up to duality.  

Also recall that the group of outer automorphisms of a compact
semisimple Lie algebra $\Lg$ is
canonically isomorphic to the group of automorphisms
of its Dynkin diagram. In fact, every automorphism of  $\Lg$
can be assumed to fix a given CSA and a given Weyl chamber, up to
inner automorphism, and then it permutes the simple roots.

\begin{lem}\label{0}
  Let $\tau\in\mathrm{Aut}(\Lg)^\sigma$ of order~$2$
  be such that $\tau|_{\mathfrak a}=\mathrm{id}$.
  Then $\tau$ induces
  an automorphism of order~$2$ of the Satake diagram that
  maps each white node either to itself or to another
  white node joined to itself by a curved arrow.
\end{lem}

\Pf Since  $\tau|_{\mathfrak a}=\mathrm{id}$, we have
$\tau(\Lk_0)=\Lk_0$. Now $\tau(\Lt)$ is a CSA of $\Lk_0$,
so by the conjugacy of CSA there is an inner automorphism 
of $\Lk_0$ mapping $\tau(\Lt)$ to~$\Lt$, say $\Lt=\Ad_k\tau(\Lt)$
for some $k\in K_0$, where $K_0$ is the centralizer of $\La$ in $K$. 
Further, by conjugacy of Weyl chambers, we may assume that 
$\tilde\tau:=\Ad_k\tau$ fixes the chosen Weyl chamber of $\Lt$. 

Now $\tilde\tau\in\mathrm{Aut}(\Lg)^\sigma$,
$\tilde\tau$ preserves the CSA $\Lh$ and its Weyl chamber,
so it maps simple roots to simple roots. 
Also, $\tilde\tau|_{\mathfrak a}=\mathrm{id}$, so if $\alpha$
is a restricted root, then $\tilde\tau\alpha=\alpha\circ\tilde\tau$ and
$\alpha$ have the same restriction
along $\La$. The desired result follows.
\EPf

\medskip

We shall use Lemma~\ref{0} as follows. Many Satake diagrams do not
admit a non-trivial automorphim that fixes every white node or maps it
to another white node joined to the first by a curved arrow. For
the corresponding symmetric space, this means that $\Phi$ must be given
as the inner automorphism induced by an element of $G$. 

\section{The classification of almost symmetric s-orbits}

In this section we run through the cases of simply-connected irreducible
symmetric spaces of compact type and use the results of
section~\ref{sec:reduct} to determine the orbits of s-representations
that are almost symmetric. 

Denote the simple restricted roots by $\alpha_1,\ldots, \alpha_q$,
where $q$ is the rank of the symmetric space, and by $\delta$ 
the highest restricted root. The fundamental restricted weights 
$\lambda_1,\ldots,\lambda_q$ as usual
are defined by $2\langle \lambda_i,\alpha_j\rangle/||\alpha_j||^2=\delta_{ij}$ for all $i$, $j$ (Kronecker delta).

\subsection{Spaces of maximal rank}\label{max-rk}

These are characterized by anyone of the following
equivalent conditions: $\mathrm{rk}\,M=\mathrm{rk}\,G$;
or $\Lk_0=0$; or uniform multiplicity~$1$. 
In the table, $q$ denotes the rank of $M$. 

\[ \begin{array}{|c|c|}
\hline  
AI & SU(q+1)/SO(q+1) \\
BI & SO(2q+1)/S(O(q+1)\times O(q))\\
CI & Sp(q)/U(q) \\
DI & SO(2q)/S(O(q)\times O(q))\\
G & G_2/SO(4) \\
FI & F_4/Sp(3)Sp(1) \\
EI & E_6/[Sp(4)/\Z_2]\\
EV & E_7/[SU(8)/\Z_2]\\
EVIII & E_8/Spin(16)\\
\hline
\end{array}\]

All nodes of the Satake diagram are white and there are no curved arrows, so 
the Satake diagram coincides with the Dynkin diagram. 
Therefore $\tau$ is of inner type, $\tau=\Ad_g$ for some 
$g\in N_G(K)$. Going through the detailed description
of bottom spaces for the spaces of maximal rank in~\cite[p.~314-5]{W2}
and noting that $g\in Z_G(\La)$, we deduce that $g\in K_0$. 
Here $K_0=\Z_2^q$, so $g$ is completely determined
by its sign on $\Lp_{\alpha_i}\cong\R$ for $i=1,\ldots,q$.  

For each $\alpha\in\Delta^+$, write $\alpha=\sum_{i=1}^qn_i(\alpha)\alpha_i$,
where the $n_i(\alpha)$ are non-negative integers. 

\subsubsection{$M$ is a most singular orbit}
Without loss of generality,
we may pass to another $s$-orbit with the same tangent space
$T_{a_0}M$ and assume that $\alpha_{i_1}(a_0)=\sqrt{-1}$ for some
$i_1$, and
$\alpha_i(a_0)=0$ for $i\neq i_1$.
Now $\Lp_\alpha\subset T_{a_0}M$ holds for some $\alpha\in\Delta^+$ 
if and only if $n_{i_1}(\alpha)>0$. In particular
$\Lp_{\alpha_i}\subset\nu_{a_0}M$ for 
$i\neq i_1$,
so $\Phi|_{\mathfrak p_{\alpha_i}}=\mathrm{id}$ for $i\neq i_1$, and,
since $\Phi$ is assumed non-trivial, $\Phi|_{\mathfrak p_{\alpha_{i_1}}}=
-\mathrm{id}$. It follows that 
$\Phi|_{\mathfrak p_\alpha}=(-1)^{n_{i_1}(\alpha)}$ for all $\alpha\in\Delta^+$.
 
There is 
exactly one $\beta\in\Delta^+$  such that $n_{i_1}(\beta)$ is nonzero and even
(as this implies that $\Lp_\beta\subset T_{a_0}M$ and $\Phi|_{{\mathfrak p}_\beta}=1$),
and it must be that $n_{i_1}(\beta)=2$ (for $n_{i_1}(\beta)\geq4$ would
entail the existence of roots $\alpha$ with $n_{i_1}(\alpha)=2$).
We come to two cases.

\paragraph{$\bullet\ \beta=\delta$}
Then $\delta-\alpha_i\not\in\Delta$ for $i\neq i_1$ 
(because $n_{i_1}(\delta-\alpha_i)=2$),
so $\langle\delta,\alpha_i\rangle=0$ for all $i\neq i_1$ (since
the $\alpha_i$-chain through $\beta$ consists of $\beta$ only), and this 
implies that $\delta$ is a multiple of the fundamental weight $\lambda_{i_1}$.
By inspection of the highest root in each case (see e.g.~\cite[p.~273]{W}),
we see that
that this happens in all cases in the table, but $A_q$. 

\[ \begin{array}{|c|c|}
  \hline
A_q & \delta=\lambda_1+\lambda_q\\
B_q & \delta= \lambda_2 \\
C_q & \delta = 2\lambda_1 \\
D_q & \delta= \lambda_2 \\
G_2 & \delta= \lambda_2 \\
F_4 & \delta= \lambda_1 \\
E_6 & \delta= \lambda_2 \\
E_7 & \delta= \lambda_1 \\
E_8 & \delta= \lambda_8 \\
\hline
\end{array} \]
For such orbits, we have seen in section~\ref{sec:invol} that 
$\tau=e^{\pi\ad_{a_0}}$ induces an almost symmetry of 
$M=Ka_0$ (cf.~eqn.~(\ref{Phi-on-palpha})). This gives
exactly $8$ examples. 

\paragraph{$\bullet\ \beta\neq\delta$}
Here $n_{i_1}(\delta)=3$. By inspection, the highest root having a coefficient
equal to~$3$ implies that 
$\Lg$ is of exceptional 
type (cf.~\cite[p.~270]{W}). In the sequel, we use Bourbaki's convention for the numbering of the 
simple roots (cf.~\cite[pp.~251-275]{B}). 

In the case of $G_2$, we have $\delta=3\alpha_1+2\alpha_2$, so 
$i_1=1$ and $\beta=2\alpha_1+\alpha_2$
($\alpha_1$ is the short root). The other roots are 
$3\alpha_1+\alpha_2$,  $\alpha_1+\alpha_2$, $\alpha_2$, $\alpha_1$
and all have $n_1(\alpha)\neq2$, so $\alpha_1(a_0)=\sqrt{-1}$, $\alpha_2(a_0)=0$
yields an example. 

In the case of $F_4$, we have $i_1=2$, 
and $\delta-\alpha_1-\alpha_2$ and $\delta-\alpha_1-\alpha_2-\alpha_3$ are two roots with $n_2=2$, so there are no examples. 

In the case of $E_6$, we have $i_1=4$, and 
$\delta-\alpha_2-\alpha_4$ and $\delta-\alpha_2-\alpha_3-\alpha_4$ are two roots with $n_4=2$, so there are no examples. 

In the case of $E_7$, we can take $i_1=3$ or $5$. In the first case 
$\delta-\alpha_1-\alpha_3$ and $\delta-\alpha_1-\alpha_3-\alpha_4$ are two roots with $n_3=2$, so we do not get an example. In the second case,
$\delta-\alpha_1-\alpha_3-\alpha_4-\alpha_5$  and
$\delta-\alpha_1-\alpha_3-\alpha_4-\alpha_5-\alpha_6$ 
are two roots with $n_5=2$, so we do not get an example either. 

In the case of $E_8$, we can take $i_1=7$ or $2$. In the first case 
$\delta-\alpha_8-\alpha_7$ and $\delta-\alpha_8-\alpha_7-\alpha_6$ are two roots with $n_7=2$, so we do not get an example. In the second case,
$\delta-\alpha_8-\alpha_7-\alpha_6-\alpha_5-\alpha_4-\alpha_2$  and
$\delta-\alpha_8-\alpha_7-\alpha_6-\alpha_5-\alpha_4-\alpha_2-\alpha_3$  
are two roots with $n_2=2$, so we do not get an example either. 

\subsubsection{$M$ is an almost singular orbit}\label{aso}
Without loss of generality, we may assume that there are 
$i_1$, $i_2$ such that 
$\alpha_i(a_0)=\sqrt{-1}$ for $i=i_1$ or $i_2$ and $\alpha_i(a_0)=0$ otherwise. 
In particular $\Lp_{\alpha_{i_1}}\oplus\Lp_{\alpha_{\alpha_{i_2}}}\subset T_{a_0}M$. 
Here $\Phi|_{\mathfrak p_{\alpha_i}}=\mathrm{id}$ for $i\neq i_1$, $i_2$, and there
are two subcases.

\paragraph{$\bullet\ \Phi|_{\mathfrak p_{\alpha_{i_1}}}=\mathrm{id}$,
$\Phi|_{\mathfrak p_{\alpha_{i_2}}}=-\mathrm{id}$}
Note that $\alpha_{i_1}$ cannot be adjacent to $\alpha_i$ for $i\neq i_2$
in the Dynkin diagram. In fact, otherwise 
$\alpha_{i_1}+\alpha_i\in\Delta^+$ for 
some $i\neq i_2$ with $\Lp_{\alpha_{i_1}+\alpha_i}\subset T_{a_0}M$ and 
$\Phi|_{\alpha_{i_1}+\alpha_i}=\mathrm{id}$, yielding coindex at least $2$,
which is a contradiction to almost symmetry of $M$. This implies
that $\alpha_{i_1}$ is an extremal node of the Dynkin diagram and $\alpha_{i_2}$ 
  is adjacent to $\alpha_{i_1}$. Again by almost symmetry of $M$,
we must have $n_{i_2}(\alpha)$ odd for all $\alpha\in\Delta^+\setminus\{\alpha_{i_1}\}$. In particular, $n_{i_2}(\delta)=1$. Glancing at the coefficients
of~$\delta$ (recall that $\alpha_{i_2}$ is adjacent to an extremal node), we obtain only the following two possibilities.

\paragraph{$-\ A_q$}
$i_1=1$, $i_2=2$ and $M$ is a principal orbit and
a circle bundle over one of the  extrinsic symmetric focal orbits
$G_2(\R^{q+1})$; the almost symmetry of $M$ coincides with the
extrinsic symmetry of the focal orbit. Since there are two focal orbits,
we obtain two different almost symmetries for $M$. 

\paragraph{$-\ B_2$}
$i_1=1$, $i_2=2$ and $M$ is a principal orbit $[SO(3)\times SO(2)]/\Z_2$. 

Since we can freely choose the signs of the action of $k$ on each
$\Lp_{\alpha_i}$, this yields exactly two examples.

\paragraph{$\bullet\ \Phi|_{\mathfrak p_{\alpha_{i_1}}}=-\mathrm{id}$,
  $\Phi|_{\mathfrak p_{\alpha_{i_2}}}=-\mathrm{id}$}
Since $\Phi$ fixes a unique tangent direction, there is a unique
$\beta\in\Delta^+$ such that $n_{i_1}(\beta)+n_{i_2}(\beta)$ is positive
and even, thus, it is equal to~$2$. We consider two possibilities.

\paragraph{$-\ \beta=\delta$} Then $\delta-\alpha_i\not\in\Delta$
if $i\neq i_1$, $i_2$. This implies that $\langle\delta,\alpha_i\rangle=0$
for all $i\neq i_1$, $i_2$, hence $\delta$ is a linear
combination of two fundamental weights $\lambda_{i_1}$, $\lambda_{i_2}$.
The only posibility is $A_q$ and
\[ a_0=\left(\begin{array}{ccccc}1&&&&\\&0&&&\\&&\ddots&&\\&&&0&\\
  &&&&-1\end{array}\right). \]
The (uniform) multiplicity must be one. This example has already appeared
above, a circle bundle over $G_2(\R^{q+1})$. 

\paragraph{$-\ \beta\neq\delta$} Since $\beta\pm\alpha_i\not\in\Delta$
for $i\neq  i_1$, $i_2$, $\beta$ and $\alpha_i$ are strongly orthogonal
for all $i\neq i_1$, $i_2$, so $\beta$ lies in the linear span 
of $\lambda_{i_1}$, $\lambda_{i_2}$. Also, $n_{i_1}(\delta)+n_{i_2}(\delta)$ is odd,
hence equals~$3$. We run through the cases. 

For $A_q$, $G_2$, $F_4$ and $E_8$, there are no distinct indices 
$i_1$, $i_2$ such that $n_{i_1}(\delta)+n_{i_2}(\delta)=3$.

For $E_6$, we can have $(i_1,i_2)=(1,2)$ or $(1,3)$, up to symmetry of the
Dynkin diagram, but then  
$\alpha_1+\alpha_2+\alpha_3+\alpha_4$ and 
$\alpha_1+\alpha_2+\alpha_3+\alpha_4+\alpha_5$ are two roots
with $n_{i_1}+n_{i_2}=2$.  

For $E_7$, we can have $(i_1,i_2)=(1,7)$ or $(2,7)$, but 
$\alpha_1+\alpha_3+\alpha_4+\alpha_5+\alpha_6+\alpha_7$, 
$\alpha_1+\alpha_2+\alpha_3+\alpha_4+\alpha_5+\alpha_6+\alpha_7$, 
$\alpha_2+\alpha_3+2\alpha_4+\alpha_5+\alpha_6+\alpha_7$ are roots;
the first two have $n_1+n_7=2$, and the last two have 
$n_2+n_7=2$.

For $B_q$, $i_1=1$ and $i_2>1$. Now $\theta_1$ always has $n_{i_1}+n_{i_2}=2$,
$\theta_1+\theta_q$ has $n_{i_1}+n_{i_2}=2$ if $i_2<q$,
and $\theta_{q-1}+\theta_q$ has $n_{i_1}+n_{i_2}=2$ if $i_2=q$.

For $C_q$, $i_1<q$ and $i_2=q$. Now $\theta_1+\theta_q$
always has $n_{i_1}+n_{i_2}=2$,
$\theta_2+\theta_q$ has $n_{i_1}+n_{i_2}=2$ if $i_1>1$,
and $\theta_1+\theta_{q-1}$ has $n_{i_1}+n_{i_2}=2$ if $i_1=1$.

For $D_q$, $i_1=1$ and $i_2<q-1$ or $1<i_1$ and $i_2=q-1$,
up to symmetry of the Dynkin diagram. 
Now $\theta_1+\theta_{q-1}$
always has $n_{i_1}+n_{i_2}=2$,
$\theta_1-\theta_{q-1}$ has $n_{i_1}+n_{i_2}=2$ if $i_1=1$ and $i_2<q-1$,
and $\theta_2+\theta_{q-1}$ has $n_{i_1}+n_{i_2}=2$ if $i_1>1$ and $i_2=q-1$.

Hence no new examples appear here.

\subsection{The group manifold case}

These are the symmetric spaces of the form $X=L\times L/\Delta_L$,
where $L$ is a compact connected simple Lie group. They can be equivalently 
characterized as having uniform multiplicity~$2$. 

The $s$-orbital foliation of $\Lp$ 
is  identified with the adjoint foliation of $\Ll$.  
Consider an orbit $M=La_0$ and assume $\Phi:\Ll\to\Ll$ 
is an almost symmetry for $M$. In particular $\Phi$ 
restricts to the identity along the Cartan subalgebra of $\Ll$,
so it is of inner type, $\Phi=\Ad_g$ for some $g\in L$ (in fact,
$g$ lies in the maximal torus of $L$). 
Since $L$ is connected, $\det\Ad_g=1$. On the other hand, the
$(-1)$-eigenspace of $\Phi$ on $T_{a_0}M$
has codimension $1$ in $T_{a_0}M$ and $M$ is even dimensional,
which forces $\det\Phi=-1$, a contradiction. So there are no
examples in the group case. 

We can push this idea further.  
Recall that  an irreducible symmetric space $G/K$ is called
\emph{of splitting rank} 
if $\mathrm{rk}\,G=\mathrm{rk}\,K+\mathrm{rk}\,M$. They can be 
equivalently characterized as having all multiplicities even. 

\begin{prop}\label{split-rk}
Suppose $X=G/K$ is a simply-connected irreducible symmetric space of compact 
type, rank bigger than one, and splitting rank.
Then no $K$-orbit in $\Lp$ is almost symmetric
with respect to an inner automorphism of $G$ defined by an element of $K$.
\end{prop}

\Pf Consider an orbit $M=Ka_0$. Then 
$T_{a_0}M=\sum_{\alpha\in\Delta^+:\alpha(a_0)\neq0}\Lp_\alpha$ is even dimensional
(since all multiplicities are even). 
If $\Phi:\Lp\to\Lp$ is an almost symmetry for $M$, then 
$\det\Phi=(-1)^{\dim M-1}=-1$. If $\Phi=\Ad_k$
is an almost symmetry of $M$, where $k\in K$, 
then $\det\Ad_k=1$ (since $K$ is connected), a contradiction. \EPf   

\subsection{The $A_q$-case}

There are four irreducible symmetric spaces of compact type
and Dynkin diagram of type $A_q$, each with uniform multiplicity~$m$,
where $m=1$, $2$, $4$, $8$, respectively, and $q=2$ if $m=8$.
The case $m=1$ is the space of maximal rank AI, 
and the case $m=2$ is the group case of type
$A$, which have both been dealt with already. There remains two spaces. 

\subsubsection{AII}

The space is $SU(2q+2)/Sp(q+1)$ and $m=4$. The effective space is
\[ [SU(2q+2)/\Z_2]/[Sp(q+1)/\Z_2]. \]
From the Satake diagram of AII, we invoke Lemmata~\ref{2} and~\ref{0}
to see that $\Phi=\Ad_g|_{\mathfrak p}$ for some $g\in G$. 
Now the conjugation $\mathrm{Inn}_g|_K$ is an automorphism of $K$, 
which must be inner
because $K$ has no outer automorphisms, that is, 
we can find $k\in K$ such that $gk^{-1}\in Z_G(K)$. But in this case 
the latter group coincides with the center $Z(G)$
(since he bottom space of $X$ is
$\frac{[SU(2q+2)/\Z_2]/\Z_{q+1}}{Sp(q+1)/\Z_2}$, where $\Z_{q+1}$ is
the center of $SU(2q+2)/\Z_2$, cf.~\cite[p.~315]{W2}), 
so $\Phi=\Ad_k|_{\mathfrak p}$.
In view of Proposition~\ref{split-rk} no orbit is 
almost symmetric.

\subsubsection{EIV}

The space is $E_6/F_4$ and $m=8$. This is an effective 
presentation and the bottom space is 
\[ \frac{E_6/\Z_3}{F_4\Z_3/\Z_3}=\frac{E_6}{F_4\Z_3}=\frac{E_6/\Z_3}{F_4}, \]
so $Z_{E_6}(F_4)=Z(E_6)=\Z_3$. 
The same argument involving
Lemmata~\ref{2} and~\ref{0}
and Proposition~\ref{split-rk} 
as in case AII implies that no orbit is 
almost symmetric. 

\subsection{Hermitian symmetric spaces}

\subsubsection{AIII} The space is $SU(p+q)/S(U(p)\times U(q))$. Suppose
$M=Ka_0$ is almost symmetric, with almost symmetry $\Phi$.
Then $\Phi|_{\mathfrak a}=\mathrm{id}$. In view
of Lemma~\ref{0} and the Satake diagram, 
there are two cases to be considered.

\paragraph{$\bullet\ \Phi$ is inner} Here $\Phi=\Ad_g$ for some $g$
in the normalizer $N_G(K)$ of $K$ in $G$. 
Consider first the case $g=k\in K$, so that
$k\in K_0\cong S(U(1)^q\times U(p-q))$. Then $\Phi$ commutes
with the complex structure $J$ on $\Lp$, therefore $\Phi|_{J\mathfrak a}=
\mathrm{id}$, that is,  
$\Phi|_{{\mathfrak p}_{2\theta_i}}=
\mathrm{id}$ for $i=1,\ldots, q$. Since we are assuming that the rank $q\geq2$,
by the almost-symmetry assumption, at most one index $i$ has
$2\theta_i(a_0)\neq0$, in fact, exactly one, say, $\theta_1(a_0)=\sqrt{-1}$
and $\theta_i(a_0)=0$ for $i>1$ (since the orthogonal 
complement of $J\La$ in $T_{a_0}M$ is a complex subspace). Now
$\alpha_1(a_0)=\sqrt{-1}$ and $\alpha_2(a_0)=\cdots=\alpha_q(a_0)=0$. The tangent space
\[ T_{a_0}M=\Lp_{\theta_1}+\Lp_{2\theta_1}+\sum_{j=2}^q\Lp_{\theta_1+\theta_j}+
\sum_{j=2}^q\Lp_{\theta_1-\theta_j}. \]
We need $\Phi=\Ad_k$ to be $\mathrm{-id}$ on all summands of $T_{a_0}M$, but the
second one, and to be $\mathrm{id}$ on $\La$ and the other
restricted root spaces. 
Suppose
\[ \Lk=\left\{\left(\begin{array}{cc}A&\\&B\end{array}\right)\;:\;A\in\mathfrak{u}(q), B\in\mathfrak{u}(p), \mathrm{tr}(A+B)=0\right\} \]
and
\begin{equation}\label{p} 
\Lp=\left\{\left(\begin{array}{cc}&Z\\-Z^*\end{array}\right)\;:\;Z\in\mathfrak{gl}(q,\C)\right\}. 
\end{equation}
  Take
  \[ \La=\sum_{i=1}^q\R(E_{i,q+i}-E_{q+i,i}). \]
Then
\begin{equation}\label{k-element}
 k=\left(\begin{array}{ccc}I_{1,q-1}&&\\&I_{1,q-1}&\\&&I_{p-q}\end{array}\right)\in K_0 
\end{equation}
does the job, where
\[ I_{1,q-1}=\left(\begin{array}{cccc}-1&&&\\&1&&\\&&\ddots&\\&&&1\end{array}\right) \]
  has size~$q$, and $I_{p-q}$ denotes the identity matrix of size $p-q$.

Next we observe that $N_G(K)\neq K$ occurs only if $p=q$, in which case 
$N_G(K)/K=\Z_2$ generated by the class of the map $\beta:x\mapsto x^\perp$
that takes a $q$-plane to its orthogonal complement in $2q$-space. 
We can represent $\beta$ by right-multiplication by the matrix
\begin{equation}\label{g0}
  g_0=\left(\begin{array}{cc}&I\\-I&\end{array}\right).
  \end{equation}
If $\Phi$ involves $g_0$, this means that 
$\Phi=\Ad_{g_0k}$ for some $k\in K_0\cong S(U(1)^q)$.
We have $\Ad_{g_0}(Z)=Z^*$ 
in regard to~(\ref{p}). The Dynkin diagram is  $C_q$,
where $\theta_i\pm\theta_j$ has multiplicity $2$ and $2\theta_i$ has 
multiplicity $1$. The group~$K_0$ is connected,
so it acts trivially on $\Lp_{2\theta_i}$. 
Now $\Phi|_{\Lp_{2\theta_i}}=-\mathrm{id}$, so 
$\Lp_{2\theta_i}\subset T_{a_0}M$ for all $i$. Also
$\Phi|_{\mathfrak p_{\theta_i\pm\theta_j}}=\mp\Ad_k|_{\mathfrak p_{\theta_i\pm\theta_j}}$, but $\dim\Lp_{\theta_i\pm\theta_j}=2$,
so it is impossible to have an almost symmetric $M=Ka_0$.

\paragraph{$\bullet\ \Phi$ is not inner} Suppose $\Phi=\sigma_0\Ad_k$ 
for some $k\in K_0$,
where $\sigma_0X=\bar X$ for $X\in\Lg$.  Then $\Phi$ anticommutes
with the complex structure~$J$ on~$\Lp$, therefore $\Phi|_{J\mathfrak a}=-\mathrm{id}$,
so that $\Phi|_{{\mathfrak p}_{2\theta_i}}=
-\mathrm{id}$ for $i=1,\ldots, q$. The restricted root spaces
\[ \Lp_{\theta_i\pm\theta_j}\ \mbox{and}\ \Lp_{\theta_i} \]
are complex and $\Phi$-invariant, with multiplicities
$2$ and $2(p-q)$. In particular $\Phi$ has 
$1$ and $-1$ as eigenvalues on $\Lp_{\theta_1\pm\theta_2}$, each  
with multiplicity $1$ (recall that $q\geq2$). This implies
that both $\Lp_{\theta_1\pm\theta_2}$ must be in $T_{a_0}M$, but then the 
codimension of the fixed point set of $\Phi$ in $T_{a_0}M$ is at least~$2$,
a contradiction. Hence there are no examples in this case. 

The other case is~$p=q$ and~$\Phi=\sigma_0\Ad_{g_0k}$, $g_0$ as
in~(\ref{g0}),
for some~$k\in K_0$. Then  $\Phi$ commutes
with the complex structure on $\Lp$ and this forces $a_0$ to be given by 
$\theta_1(a_0)=\sqrt{-1}$ and $\theta_i(a_0)=0$ for $i>1$, up to renaming the 
indices, by the same argument as in the case $\Phi$ is inner.
But we saw above
that this s-orbit is almost symmetric with respect to $\Phi$ of the 
form $\Ad_k$ where $k$ is given by~(\ref{k-element}). 

\subsubsection{EVII} The space is $X=E_7/\{[E_6\times U(1)]/\Z_3\}$.
Suppose
$M=Ka_0$ is almost symmetric, with almost symmetry $\Phi$.
Then $\Phi|_{\mathfrak a}=\mathrm{id}$. 
According to~\cite[p.~314]{W2}, the bottom space is the quotient
by a $\Z_2$-group which is not contained in $E_7$. Therefore 
$N_G(K)=K$. In view
of Lemma~\ref{0} and the Satake diagram, 
$\Phi=\Ad_k$ for some $k\in K_0\cong Spin(8)$. 

Now $\Phi$ commutes with the complex structure~$J$ on~$\Lp$. 
The Dynkin diagram is $C_3$ and $J(\La)$ is spanned by 
$\Lp_{\theta_i}$ for $i=1$, $2$, $3$ (each one of dimension~$1$). 
Therefore $\Lp_{\theta_i}\subset T_{a_0}M$ for 
at most one index, say, $1$. 
Thus we can write
\[ T_{a_0}M = \Lp_{\theta_1+\theta_2}+\Lp_{\theta_1-\theta_2}+\Lp_{\theta_1+\theta_3}+\Lp_{\theta_1-\theta_3}+\Lp_{\theta_1}, \]
where $\Phi=-\mathrm{id}$ on each summand, but the last one where it is 
$\mathrm{id}$. 

It is known that $K_0$ acts on $\Lp_{\theta_i+\theta_j}$ and 
$\Lp_{\theta_i-\theta_j}$ by the same one of the $8$-dimensional 
irreducible representations of $Spin(8)$, and in fact,
it acts on the triple
$(\Lp_{\theta_1+\theta_2}$, $\Lp_{\theta_2+\theta_3}$, $\Lp_{\theta_1+\theta_3})$
by three inequivalent $8$-dimensional 
irreducible representations of $Spin(8)$, say $(\rho_8, \Delta_8^+,\Delta_8^-)$. 
Let $k$ be the nontrivial element in $\ker\Delta_8^+\cong\Z_2$. Then 
$\rho_8(k)=\Delta_8^-(k)=-\mathrm{id}$ acts on $\Lp$ as wished. 
This produces the example of almost symmetric 
submanifold $M=Ka_0\cong E_6/Spin(10)$. It is also a 
homogeneous CR-manifold.

\subsubsection{EIII} The space 
is $X=E_6/\{[Spin(10)\times U(1)]/\Z_4\}$ and this is also
the bottom space~\cite[p.~314]{W2}. In particular $N_G(K)=K$. 
Suppose
$M=Ka_0$ is almost symmetric, with almost symmetry $\Phi$.
Then $\Phi|_{\mathfrak a}=\mathrm{id}$. In view
of Lemma~\ref{0} and the Satake diagram, 
there are two cases, namely, 
$\Phi=\Ad_k$ or $\Phi=\tau\Ad_k$, for some
$k\in K_0\cong U(4)$, where $\tau$ is the outer 
involution of $E_6$ with fixed point set $F_4$. 
It follows from~\cite[Example~3.5]{K} that $\tau$ has eigenvalues 
$\pm1$ on $\Lp$ with multiplicities $16$ and $16$. In particular, $\tau$
preserves the orientation of $X$. Since $K_0$ is connected,
in any case~$\Phi$ is orientation-preserving. It follows that the 
dimension of $M$ must be odd. But $X$ is a rank~$2$ symmetric space,
and the s-orbits have dimensions $30$, $24$ and $21$, respectively. 
Therefore $M$ must be the $21$-dimensional 
singular orbit $Spin(10)/SU(5)$ (cf.~\cite[p.~77]{Ko}). 

We next show that there is $k\in K_0$ such that $\Phi=\Ad_k$ 
induces an almost symmetry on that singular orbit. The Dynkin diagram
of~$X$ is~$BC_2$, and~$M=Ka_0$, where we may assume
$\theta_1(a_0)=\sqrt{-1}$, $\theta_2(a_0)=0$.
Now
\[ T_{a_0}M = \Lp_{\theta_1+\theta_2}+\Lp_{\theta_1-\theta_2}+\Lp_{\theta_2}+
\Lp_{2\theta_2}, \]
where the summands have dimensions $6$, $6$, $8$ and $1$, respectively,
and $\Phi=-\mathrm{id}$ on each summand, but the last one where
$\Phi=\mathrm{id}$. It is known that $K_0\cong U(4)$ acts as
$(SU(4),\C^4)$ on $\Lp_{\theta_2}$, as $(SO(6),\R^6)$ on
$\Lp_{\theta_1\pm\theta_2}$, and trivially on $\Lp_{2\theta_2}$. 
Let $k=(\sqrt{-1}I_4,\sqrt{-1}I_4)\in U(1)\times SU(4)$. 
Since $\sqrt{-1}I_4$ maps to $I_6$ under the double covering $SU(4)\to SO(6)$,
$k$ acts as $-1$ on
$\Lp_{\theta_1\pm\theta_2}\cong\R^6$ and on~$\Lp_{\theta_2}\cong\C^4$,
as wished.

\subsubsection{BDI} The space is $X=SO(p+2)/SO(2)\times SO(p)$
and the bottom space 
is $SO(p+2)/S(O(2)\times O(p))$~\cite[p.~314]{W2}, 
that is $N_G(K)=S(O(2)\times O(p))$,
where BI (resp.~DI) corresponds to $p$
odd (resp.~even). 
Suppose
$M=Ka_0$ is almost symmetric, with almost symmetry $\Phi$.
Then~$\Phi|_{\mathfrak a}=\mathrm{id}$.

\paragraph{$\bullet\ p$ is odd} In view
of Lemma~\ref{0} and the Satake diagram, $\Phi=\Ad_g$ for some
$g\in Z_{N_G(K)}(\La)\cong \Z_2^2\times SO(p-2)$. 
Suppose first that $g=k\in K_0=Z_K(\La)\cong \Z_2 \times SO(p-2)$. 
Then $\Phi$ commutes
with the complex structure on $\Lp$, so $\Phi|_{\mathfrak p_{\theta_1\pm\theta_2}}=
\mathrm{id}$. It follows that exactly one of $\theta_1\pm\theta_2$
in nonzero on $a_0$, say $\theta_1(a_0)=\theta_2(a_0)=\sqrt{-1}$.
Now
\[ T_{a_0}M=\Lp_{\theta_1}+\Lp_{\theta_2}+\Lp_{\theta_1+\theta_2}, \]
where $\Phi=-\mathrm{id}$ on the first two summands, and
$\Phi=\mathrm{id}$ on the last summand.

Suppose
\[ \Lk=\left\{\left(\begin{array}{cc}A&\\&B\end{array}\right)\;:\;A\in\mathfrak{so}(2), B\in\mathfrak{so}(p), \right\}. \]
  Take
  \[ \La=\sum_{i=1}^2\R(E_{i,2+i}-E_{2+i,i}). \]
Then
\begin{equation}\label{k2}
  k=I_{4,p-2}=\left(\begin{array}{cc}-I_4&\\&I_{p-2}\end{array}\right)\in K_0
  \end{equation}
does the job and yields an example of almost symmetric 
submanifold (this element generates the $\Z_2$-factor of $K_0$). In fact
\[ \Lp_{\theta_i}=\sum_{j=1} ^{p-2}\R(E_{i,4+j}-E_{4+j,i}), \]
and
\[ \Lp_{\theta_1\pm\theta_2}=\R(E_{14}\mp E_{23})-(E_{41}\mp E_{32}). \]

Next we consider the case in which $g\in N_G(K)\setminus K$; we may assume
$g=g_0k$, where
\[ g_0=\left(\left(\begin{array}{cc}-1&\\&1\end{array}\right),
\left(\begin{array}{cc}-1&\\&I_{p-1}\end{array}\right)\right)\in S(O(2)\times O(p))
\]
and $k\in K_0$.
In this case $\Phi$ anticommutes with the complex structure on $\Lp$,
so $\Phi|_{\mathfrak p_{\theta_1\pm\theta_2}}=-\mathrm{id}$. 
It follows that $\Lp_{\theta_1\pm\theta_2}\subset T_{a_0}M$. If one 
of $\Lp_{\theta_1}$, $\Lp_{\theta_2}$ is contained in $\nu_{a_0}M$,
say, $\Lp_{\theta_1}\subset\nu_{a_0}M$, then $\Phi|_{\mathfrak p_{\theta_1}}=\mathrm{id}$.
Since $J\Lp_{\theta_1}=\Lp_{\theta_2}$, this implies that 
$\Phi|_{\mathfrak p_{\theta_2}}=-\mathrm{id}$ and $\Lp_{\theta_2}\subset T_{a_0}M$.
This gives $\Phi=-\mathrm{id}$ on $T_{a_0}M$, a contradiction. 
We deduce that $M$ is a principal orbit. But we saw above
that $\Phi|_{\mathfrak p_{\theta_1\pm\theta_2}}=-\mathrm{id}$, and $\Phi$ 
has opposite signs on $\Lp_{\theta_1}$ and $\Lp_{\theta_2}$ (both of 
dimension~$p-2$). The only possibility is $p=3$, which gives a 
space of maximal rank, already studied in subsection~\ref{max-rk}.

\paragraph{$\bullet\ p$ is even} Here $\Phi=\Ad_g$
or $\Phi=\tau\Ad_g$ for some $g\in N_G(K)$,
where $\tau$ is the outer automorphism of $\Lg$ given by
conjugation by the matrix $g_0=I_{1,p+1}$.
Since $p$ is even, $K_0\cong\Z_2\times SO(p-2)$ acts
on $\Lp$ preserving the orientation, 
and $\det(\tau:\Lp\to\Lp)=\det(\Ad_{g_0}:\Lp\to\Lp)=(-1)^p=1$. 
Therefore $\Phi$ is orientation preserving.
It follows that $\dim M$ is odd. Since $X$ has dimension $2p$ and rank $2$,
with multiplicities $1$ and $p-2$, we deduce that again $M$
must the singular orbit $Ka_0$ with
$\theta_1(a_0)-\theta_2(a_0)=\alpha_1(a_0)=0$.
And in fact $\Phi=\Ad_k$ with the same $k\in K_0$ given as in~(\ref{k2})
defines an almost symmetry for this~$M$. 

We deduce that, for all $p\geq3$, the singular orbit
(Stiefel manifold)~$SO(p)/SO(p-2)=V_2(\R^p)$
is an almost symmetric submanifold.

\subsubsection{DIII} The space is $X=SO(2n)/U(n)$, where $n=2q$
or $2q+1$. 
Suppose
$M=Ka_0$ is almost symmetric, with almost symmetry $\Phi$.

The Dynkin diagram is~$C_q$ or $BC_q$ according to whether $n$ is even or odd,
and the Satake diagram also differs. 
According to~Lemma~\ref{0} and the Satake diagram,
if $n$ is even then 
$\Phi=\Ad_g$ for some
$g\in N_G(K)$. Here $N_G(K)$ generated by
$g_0=I_{n,n}$ and $K_0\cong SU(2)^q$. On the other hand, if $n$ is odd then
$N_G(K)=K$, but there is an admissible  symmetry of the Satake diagram,
so $\Phi=\Ad_k$ or $\Phi=\tau\Ad_k$,
for some
$k\in K_0\cong SU(2)^q\cdot U(1)$,
where $\tau=\Ad_{I_{n,n}}$ is an outer automorphism of $\mathfrak{so}(2n)$.
Thus in either case $\Phi=\Ad_k$ or $\Phi=\tau\Ad_k$, where
$k\in K_0$ and $\tau=\Ad_{I_{n,n}}$. Note that $\tau$ preserves the restricted
root spaces. 

\paragraph{$\bullet\ \Phi=\Ad_k$} 
Then $\Phi$ commutes with the
complex structure~$J$ on $\Lp$, so $\Phi=\mathrm{id}$ on
$J(\La)=\sum_{i=1} ^q\Lp_{2\theta_i}$. Therefore for exactly one index~$i$
we have $\Lp_{2\theta_i}\subset T_{a_0}M$ (the multiplicity of all $2\theta_i$
is $1$). Recall that
\[ \alpha_1=\theta_1-\theta_2,\ldots,\alpha_{q-1}=\theta_{q-1}-\theta_q,\
\alpha_q=\left\{\begin{array}{ll}2\theta_q,&\mbox{if $q$ is even,}\\
\theta_q,&\mbox{if $q$ is odd.} \end{array}\right. \]
Without loss of generality,
we may assume that $\theta_1(a_0)=\sqrt{-1}$ and $\theta_2(a_0)=\cdots=\theta_q(a_0)=0$,
that is, $\alpha_1(a_0)=1$ and $\alpha_2(a_0)=\cdots=\alpha(a_0)=0$. Now
\[ T_{a_0}M = \Lp_{2\theta_1}+ \sum_{j=2}^q \Lp_{\theta_1+\theta_j}+\sum_{j=2}^q
\Lp_{\theta_1-\theta_j}, \]
if $n$ is even, and
\[ T_{a_0}M = \Lp_{2\theta_1}+ \Lp_{\theta_1}+\sum_{j=2}^q \Lp_{\theta_1+\theta_j}+\sum_{j=2}^q
\Lp_{\theta_1-\theta_j}, \]
if $n$ is odd, 
where in each case 
$\Phi=\mathrm{id}$ on the first summand, and $\Phi=-\mathrm{id}$
on the other summands.

Now we take $k=(-1,1,\ldots,1)\in K_0\cong Sp(1)^q$
(resp.~$k=(-1,1,\ldots,1;1)\in Sp(1)^q\cdot U(1)$)
and note that
$\Ad_k$ acts on $\Lp$ as wished, since
$(k_1,\ldots,k_q)\in K_0\cong Sp(1)^q$
(resp.~$(k_1,\ldots,k_q,t)\in K_0\cong Sp(1)^q\cdot U(1)$)
acts on $\Lp_{2\theta_i}\cong\R$ trivially
and on $x\in\Lp_{\theta_i\pm\theta_j}\cong \Q$ as $k_ixk_j^{-1}$
(quaternionic multiplication) in both cases; and 
on $y\in\Lp_{\theta_i}\cong\C^2$ as $k_iyt^{-1}$ in case $n$ is odd.
Therefore $M\cong U(2q)/SU(2)\times U(2q-2)$
(resp.~$U(2q+1)/SU(2)\times U(2q-1)$)
is an almost symmetric submanifold.

\paragraph{$\bullet\ \Phi=\tau\Ad_k$}
Recall that
\[ \Lk=\left\{\left(\begin{array}{cc}A&B\\-B&A\end{array}\right): A\in\mathfrak{so}(n),\ B^t=B\right\}, \]
and
\[ \Lp=\left\{\left(\begin{array}{cc}A&B\\B&-A\end{array}\right): A,\ B\in\mathfrak{so}(n)\right\}. \]
The complex structure on $\Lp$ is given by 
\[ J\left(\begin{array}{cc}A&B\\B&-A\end{array}\right)
=\left(\begin{array}{cc}-B&A\\A&B\end{array}\right). \]
It follows that $\tau$ anti-commutes with $J$, and so does $\Phi$. 
Therefore $\Phi|_{\mathfrak p_{2\theta_i}}=-\mathrm{id}$, so
$\sum_{i=1}^q\Lp_{\theta_{2i}}+\sum_{i=1}^q\Lp_{\theta_i}\subset T_{a_0}M$.
This implies that for each pair $i$, $j$ at least one
of $\Lp_{\theta_i\pm\theta_j}$ is contained in $T_{a_0}M$.

The description above of $\Lp_{\theta_i\pm\theta_j}$ as the $K_0$-representation
of real type
\[ (SO(4)=Sp(1)_{(i)}Sp(1)_{(j)},\R^4=\Q\otimes_{\mathbb H} \Q^*) \]
shows that $J$ must map $\Lp_{\theta_i\pm\theta_j}$ to~$\Lp_{\theta_i\mp\theta_j}$
(this can also be checked directly from the formulae in~\cite[p.~116]{Loos2}).

Fix $i,j$ and note that $\Lp_{\theta_{i}+\theta_{j}}+\Lp_{\theta_{i}-\theta_{j}}$
is a  $\Phi$-invariant complex subspace of
real dimension $8$. Therefore $\Phi$ has $4$ eigenvalues $+1$
and $4$ eigenvalues $-1$ there. Thus it must be that
one of~$\Lp_{\theta_i\pm\theta_j}$ is contained in $\nu_{a_0}M$
and the other one in $T_{a_0}M$. Recall that $\Phi$ anticommutes with $J$,
so if $n$ is even this means that
$\Phi=-\mathrm{id}$ on $T_{a_0}M$ which is a contradition.
If $n$ is odd, this means that $\Phi$ can still have a $(+1)$-eigenvalue
on~$\Lp_{i_0}$ ($\subset T_{a_0}M$)
for some $i_0$. But $\Lp_{i_0}$ is a $\Phi$-invariant 
complex subspace of
real dimension $4$, therefore $\Phi$ has $2$ eigenvalues $+1$
and $2$ eigenvalues $-1$ there, which is also a contradiction. 
Hence we get no other examples in this case.

\subsection{Real Grassmannians ($BDI$)}

The space is the Grassmannian of oriented planes
$X=SO(p+q)/SO(p)\times SO(q)$, and the
bottom space is the Grassmannian of unoriented planes
$SO(p+q)/S(O(p)\times O(q))$. The transvection group
is $SO(p+q)$ if $p$ or $q$ is odd, and $SO(p+q)/\Z_2$
if $p$ and $q$ are even. We may assume $p\geq q\geq3$
($q=2$ yields a Hermitian
symmetric space), and $p\neq q$, $q+1$
(for otherwise $X$ is a space of maximal rank).
Then the Dynkin diagram is~$B_q$. 
We suppose
\[ \Lk=\left\{\left(\begin{array}{cc}A&\\&B\end{array}\right)\;:\;A\in\mathfrak{so}(q), B\in\mathfrak{so}(p), \right\}, \]
and take
  \[ \La=\sum_{i=1}^q\R(E_{i,q+i}-E_{q+i,i}). \]
Assume 
$M=Ka_0$ is almost symmetric, with almost symmetry~$\Phi$.
Then $\Phi|_{\mathfrak a}=\mathrm{id}$.

\paragraph{$\bullet\ \Phi$ is inner} 
In view
of Lemma~\ref{0} and the Satake diagram, this must be the case 
if~$n=p+q$ is odd. 
So  $\Phi=\Ad_g$ for some
$g\in Z_{N_G(K)}(\La)\cong \Z_2^q\times SO(p-q)$. By relabeling the entries, we
may assume that
\[ g =\left(\begin{array}{ccc}I_{r,s}&&\\&I_{r,s}&\\&&A\end{array}\right), \]
  where $A\in SO(p-q)$ ($A$ has order~$2$, so it can be further assumed
  to have only diagonal entries $\pm1$, up to conjugation).
  It follows that $\Phi|_{\mathfrak p_{\theta_i\pm\theta_j}}=\mathrm{id}$ if $1<i<j\leq r$ or $r+1\leq i< j\leq q$, and
$\Phi|_{\mathfrak p_{\theta_i\pm\theta_j}}=-\mathrm{id}$ otherwise.
  In particular, at most one of
  $\Lp_{\theta_i+\theta_j}$ or~$\Lp_{\theta_i-\theta_j}$
for some $i<j\leq r$ or~$r+1\leq i<j$ is contained in $T_{a_0}M$. 
It is clear that at least one of them must be contained in $T_{a_0}M$ 
(otherwise all $\theta_i(a_0)=0$). 
By relabeling the indices, we may assume 
$\theta_1+\theta_2(a_0)\neq0$, $\theta_1-\theta_2(a_0)=0$ 
and $\theta_i\pm\theta_j(a_0)=0$ if $i<j\leq r$ and $(i,j)\neq(1,2)$,
or $r+1\leq i<j$. And $\theta_i\pm\theta_j(a_0)\neq0$ if $i\leq r<j$. 
The only solution, up to scaling, is 
\[ r=2, \ \theta_1(a_0)=\theta_2(a_0)=\sqrt{-1},\ \theta_3(a_0)=\cdots\theta_q(a_0)=0. \]
Now 
\[ T_{a_0}M=\Lp_{\theta_1+\theta_2}+\Lp_{\theta_1}+\Lp_{\theta_2}
+\sum_{j=3}^q\Lp_{\theta_1+\theta_j} 
+\sum_{j=3}^q\Lp_{\theta_1-\theta_j} 
+\sum_{j=3}^q\Lp_{\theta_2+\theta_j} 
+\sum_{j=3}^q\Lp_{\theta_2-\theta_j}, \] 
where $\Phi=\mathrm{id}$ on the first summand, and
$\Phi=-\mathrm{id}$ on the other summands. 

Note that
\begin{equation}\label{k}
  g=\left(\begin{array}{ccc}I_{2,q-2}&&\\&I_{2,q-2}\\&&I_{p-q}\end{array}\right)
  \end{equation}
indeed does the job, so we get an almost symmetric submanifold
$M\cong S(O(p)\times O(q))/S(O(2)\times O(p-2)\times O(q-2))$. 

\paragraph{$\bullet\ \Phi$ is not inner} 
This can happen if $n$ is even. The only difference from above is that 
$A\in O(p-q)$, but this does not affect the above argument. 

\subsection{Quaternion-K\"ahler symmetric spaces}

We prove the next proposition just in the degree of generality 
needed for the sequel.

\begin{prop}\label{qK}
Let $X'$ be a reflective submanifold 
of maximal rank of a symmetric space $X$ of exceptional type,
with $o\in X'$. 
If $M$ is an almost symmetric 
s-orbit in $T_oX$, then $M'=M\cap T_oX'$ is either an
almost symmetric or a symmetric s-orbit in $T_oX'$.
\end{prop}

\Pf Let $\Lg=\Lk+\Lp$ be the involutive decomposition under $\sigma$ 
of the Lie algebra $\Lg$ of the transvection group of $X$. 
The reflective property of $X'$ implies that
there is an involutive automorphism $\tau$ of $\Lg$ 
that commutes with $\sigma$ such that the fixed point set $\Lp^\tau=T_oX'$,
and $\Lg^\tau=\Lk^\tau+\Lp^\tau$ is the involutive Lie algebra
associated to the symmetric space~$X'$. 
Since $X'$ has maximal rank, we can take a common Cartan
subspace~$\La\subset\Lp^\tau\subset\Lp$. Write as usual
\[ \Lk=\Lk_0+\sum_{\alpha\in\Delta^+}\Lk_\alpha,\ \Lp=\La+\sum_{\alpha\in\Delta^+}\Lp_\alpha. \]
Consider the $s$-orbit~$M$ of $X$ through a given point $a\in\La$,
and a non-zero 
tangent vector $v=[x,a]\in\Lp_\alpha$ at $a$, where $x\in\Lk_\alpha$. 
In particular, $\alpha(a)\neq0$. 
Assume that $v\in T_oX'$. Then $\tau v=v$, so $\tau x-x$ lies in the 
isotropy algebra $\Lk_a$. Since $\tau|_{\mathfrak a}=\mathrm{id}$, 
we have $\tau(\Lk_\alpha)=\Lk_\alpha$, so in fact $\tau x-x\in(\Lk_\alpha)_a$.
Owing to $\alpha(a)\neq0$, the latter isotropy algebra is zero, which yields
$\tau x=x$. Hence $x\in\Lk^\tau$ and $v$ is tangent to the the $s$-orbit
of $X'$ through~$a$. This shows that $M'$ is the s-orbit of $X'$ through~$a$.

Recall that 
\[ T_aM=\sum_{\alpha:\alpha(a)\neq0}\Lp_\alpha,\
\nu_aM=\La+\sum_{\alpha:\alpha(a)=0}\Lp_\alpha. \]
Since $\tau|_{\mathfrak a}=\mathrm{id}$, we have $\tau(\Lp_\alpha)=\Lp_\alpha$ 
for all $\alpha$. In particular 
\[ T_aM=\underbrace{T_aM^\tau}_{=T_aM'} \oplus T_aM^{-\tau},\ \nu_aM=\nu_aM^\tau \oplus \nu_aM^{-\tau}. \]

Let $\Phi$ be an almost symmetry 
for $M$. Then $\Phi=\nu|_{\mathfrak p}$ 
for some involutive automorphism $\nu$ of $\Lg$. 
Since $X$ is of exceptional type, we may 
conjugate $\nu$ and assume that $\nu$ and $\tau$ commute~\cite{C}. 
It follows that the $\pm1$-eigenspaces of $\nu$ on $\Lp$ 
decompose with respect to $\Lp=\Lp^\tau+\Lp^{-\tau}$. 
Note that $\nu=\mathrm{id}$ on $\nu_aM^\tau=\nu_aM'\subset\Lp^\tau$. 
Next there are two cases: either $\nu=-\mathrm{id}$ on $T_aM'$ 
or $\nu$ has an eigenvalue~$1$ of multiplicity~$1$ in $T_aM'$;
these two cases corresponding accordingly to $M$' being 
symmetric or almost symmetric. This completes the proof of the 
proposition. \EPf

\subsubsection{Exceptional spaces}

There is one irreducible qK symmetric space for each simple Lie group.
We first consider the spaces of exceptional type. Note that $G_2/SO(4)$
and $F_4/[Sp(3)Sp(1)]$ are  spaces of 
maximal rank and thus they have already been dealt with; in particular, 
$F_4/[Sp(3)\cdot Sp(1)]$ gives rise to exactly one
almost symmetric submanifold. 
Next, there is a chain of reflective submanifolds:
\[  F_4/[Sp(3)\cdot Sp(1)]\subset E_6/[SU(6)\cdot SU(2)]
\subset E_7/[Spin(12)\cdot SU(2)] \subset E_8/[E_7\cdot SU(2)], \]
or $FI \subset EII\subset EVI \subset EIX.$
All spaces have rank~$4$ and indeed Dynkin diagram $F_4$. 
The almost symmetric s-orbit of $FI$ is
the s-orbit through $a_0\in\La$ defined by
\[ \alpha_1(a_0)=\sqrt{-1},\ \alpha_2(a_0)=\alpha_3(a_0)=\alpha_4(a_0)=0. \]
And there are no symmetric s-orbits in $FI$ (cf.~\cite{ET}). 
It follows from Proposition~\ref{qK} that, in the other three cases,
it is only the s-orbit through the
same point $a_0\in\La$ that has a chance of being almost symmetric.
That orbit is indeed almost symmetric in those three cases,
because the almost
symmetry for $FI$ is given by~$\Phi=e^{\pi\mathrm{ad}_{a_0}}|_{\mathfrak p}$,
and this clearly
extends to an almost symmetry in the other three spaces (as
in all cases $\Phi|_{{\mathfrak p}_\alpha}=(-1)^{n_1(\alpha)}$,
the root system is $F_4$, and the only $\alpha\in\Delta^+$ for which
$n_1(\alpha)$ is even and nonzero is the highest root $\delta$,
which has multiplicity one).

\subsubsection{Quaternionic Grassmannians}

We consider the case $CII$, that is,
$X=Sp(p+q)/[Sp(p)\times Sp(q)]$, where $p\geq q\geq2$.
This is also the bottom space, except in the case $p=q$ in which
the bottom space is the $\Z_2$-quotient by the map $\beta:x\mapsto x^\perp$
that takes a $q$-plane to its orthogonal complement in $2q$-space.
There are no symmetries of the Satake diagram, so we need
to consider the following two cases only. 

\paragraph{$\bullet\ \Phi$ does not involve $\beta$} Here $\Phi=\Ad_k$ for some
$k\in K_0\cong Sp(1)^q\times Sp(p-q)$ of order~$2$. It follows
that $k=(\pm1,\ldots,\pm1,A)$ for some $A\in Sp(p-q)$.
In particular $\Phi|_{\Lp_{2\theta_i}}=\mathrm{id}$ for all $i=1,\ldots,q$.
Now for at least one index $i$ we must have $\Lp_{2\theta_i}\subset T_{a_0}M$,
and $\Lp_{2\theta_i}$ has dimension~$3$, so there cannot exist an almost
symmetric~$M=Ka_0$.

\paragraph{$\bullet\ \Phi$ involves $\beta$} Here $p=q$ and $\Phi=\Ad_{g_0k}$
for some $k\in K_0$ and $g_0$ as in~(\ref{g0}). 
We suppose
\[ \Lk=\left\{\left(\begin{array}{cc}A&\\&B\end{array}\right)\;:\;A\in\mathfrak{sp}(q), B\in\mathfrak{sp}(q), \right\}, \]
\[ \Lp=\left\{\left(\begin{array}{cc}&Z\\-Z^*&\end{array}\right)\;:\;Z\in\mathfrak{gl}(q,\Q), \right\}, \]
where $Z^*=\bar Z^t$, 
and take
\[ \La=\sum_{i=1}^q\R(E_{i,q+i}-E_{q+i,i}). \]
We have $\Ad_{g_0}(Z)=Z^*$. 

The Dynkin diagram is $C_q$,
where $\theta_i\pm\theta_j$ has multiplicity $4$ and $2\theta_i$ has 
multiplicity $3$. Now $\Phi|_{\Lp_{2\theta_i}}=-\mathrm{id}$, so 
$\Lp_{2\theta_i}\subset T_{a_0}M$ for all $i$. Also
$\Phi|_{\mathfrak p_{\theta_i\pm\theta_j}}=\mp\Ad_k|_{\mathfrak p_{\theta_i\pm\theta_j}}$,
but $\dim\Lp_{\theta_i\pm\theta_j}=4$,
so it is impossible to have an almost symmetric $M=Ka_0$. 

\medskip

This completes the case-by-case analysis of irreducible
symmetric spaces of compact type and the proof of part~(a)
of Theorem~\ref{thm:classif}.

\section{Classification in the case of non-s-orbits}

Herein we deal with almost symmetric homogeneous submanifolds
of Euclidean space that are not orbits of s-representations. 

Let $M$ be a full irreducible compact almost symmetric
homogeneous submanifold of an Euclidean space $V$ which is not an
s-orbit. Owing to Theorem~\ref{thm:struct}(a), $M$ has codimension~$3$.
Denote by $G$ the maximal connected subgroup of $SO(V)$ 
that preserves~$M$. Since representations of cohomogeneity at most two 
are always polar, and polar representations have the same orbits
as an s-representation~\cite{D}, the cohomogeneity of $(G,V)$ must be $3$. 
Now $M$ is a principal orbit of $(G,V)$ and $G$ is the maximal connected
group in its orbit-equivalence class. 

Fix $p\in M$ and let $\sigma$ denote the (unique) almost symmetry at~$p$. 
Write $\hat G$ for the identity component of the group generated by $G$ 
and $\sigma$. By construction of $G$, we have $\hat G=G$, so $\sigma$ 
normalizes $G$. Now conjugation by $\sigma$ defines an involutive 
automorphism of $G$ and hence specifies a possibly non-effective 
symmetric pair $(G,G^\sigma)$. By uniqueness of $\sigma$ we have that 
$k\sigma k^{-1}=\sigma$ for all $k\in G_p$, so the fixed point group $G^\sigma$ 
contains the isotropy group $G_p$, and indeed with codimension one, as
the $(+1)$-eigenspace of $\sigma$ on $T_pM$ is $1$-dimensional.

We run through the list of representations of cohomogeneity~$3$;
the maximal connected Lie groups in their equivalence class are listed
in~\cite[Table~II]{S}. All such representations have
\emph{copolarity} one (see~\cite{GOT} for the irreducible case
and~\cite{PP} for the reducible case), namely, there is a
$4$-dimensional subspace $\Sigma$ of the representation space meeting all
orbits and containing the normal spaces to the principal orbits
it meets. The reflection on $\Sigma$ is a candidate for the
almost symmetry of a principal orbit,
but it preserves the $G$-orbits precisely in the three
cases listed in the theorem, as we now explain.

\subsection{$(SO(n),\R^n\oplus\R^n)$, $n\geq3$}
A principal isotropy group is the lower block $SO(n-2)$-subgroup, and its
fixed point set is the upper $\R^2\oplus\R^2$ and can be taken as $\Sigma$.
The reflection on $\Sigma$ is given by $\mathrm{diag}(1,1,-1,\ldots,-1)$
and normalizes the group, hence it preserves the orbits (since it is
the identity along $\Sigma$ and $\Sigma$ meets all orbits).
It follows that every principal orbit is almost symmetric.

\subsection{$(U(2),V=\C^2\oplus\R^3)$} It is convenient to
replace $U(2)$ by $G=U(1)\times Sp(1)$
and describe this action in terms of quaternions as
\[ (e^{i\theta},q)\cdot(x,y)=(qxe^{-i\theta},qyq^{-1}), \]
where $x\in\Q$, $y\in\Im\Q$, $q\in Sp(1)$, $\theta\in\R$. 
We identify $\C^2\cong\Q$ by mapping $(\alpha,\beta)$ to $\alpha+j\beta$
so that complex conjugation corresponds to conjugation by~$j$.

The slice representation at $(x,y)=(1,0)$ is the circle group
$\{(e^{i\theta},e^{i\theta})\}\subset G$ acting on $\Im H$ as rotation on
the $jk$-plane, so it is clear that every $G$-orbit has a representative
of the form $(x,y)$ with $x\in\R$ and $y\in \R i \oplus \R k$. 

Define
\[ \tau(x,y) = (jxj^{-1},-jyj^{-1})= (-jxj,jyj). \]
Then $\tau$ is an involutive isometry of~$V$ that
normalizes~$G$ and whose fixed point set
\[ V^\tau= \{(x,y)\in V\;|x\in \R\oplus\R j,\ y\in\R i \oplus\R k\}. \]
It follows that $\tau$ preserves the orbits.
Now we can take $\Sigma=V^\tau$ and every
principal orbit is almost symmetric.

\subsection{$(G=U(1)\times SU(n) \times U(1), V=\C^n\oplus\C^n)$, $n\geq2$}.
We describe the action as
\[ (\lambda,A,\mu)\cdot(x,y)=(\lambda Ax, \mu Ay), \]
where $\lambda$, $\mu\in U(1)$, $A\in SU(n)$, $x$, $y\in \C^n$.
We have the isotropy group
\[ G_{(e_1,0)}=\{\lambda^{-1}, \left(\begin{smallmatrix}\lambda&\\&B
  \end{smallmatrix}\right), \mu)\;|\; B\in U(n-1),
\lambda^{-1}=\det B, \mu\in U(1)\}, \]
so that the iterated isotropy group
\begin{eqnarray*}
  (G_{(e_1,0)})_{(0,e_1+e_2)}&=&\{(\lambda^{-1}, \left(\begin{smallmatrix}\lambda&&\\&\lambda&\\&&C\end{smallmatrix}\right), \lambda^{-1})\;|\; C\in U(n-2),
    \lambda^{-2}=\det C \},\\
    &\cong &\Z_2\times U(n-2)
\end{eqnarray*}
is a principal isotropy group for $n\geq3$, but the principal
isotropy group is trivial for $n=2$.

Any symmetric subalgebra of $\Lg=\mathfrak{u}(1)\oplus\mathfrak{su}(n)\oplus\mathfrak{u}(1)$ has to split into symmetric subalgebras of the factors.
For $n\geq3$,
it is clear that $\R\oplus\Lg_{(e_1,e_1+e_2)}\cong\R\oplus\mathfrak{u}(n-2)$
cannot be a symmetric subalgebra of $\Lg$. 

On the other hand, for $n=2$ we consider the compĺex conjugation of
$\C^2\oplus\C^2$ over $\R^2\oplus\R^2$, denoted by $\sigma$. It is clear that
$\sigma$ normalizes $G$. Since $\Sigma=\R^2\oplus\R^2$ meets every
$G$-orbit, $\sigma$ preserves the $G$-orbits and hence every
principal orbit is almost symmetric. 

\subsection{The other cases}
These are ruled
out case-by-case by checking that there is no symmetric subalgebra
of the form $\Lg_p\oplus\R$ of $\Lg$.

\section{Inhomogeneous almost symmetric submanifolds}\label{inhomog}

In this section we discuss some results about
the inhomogeneous case. In particular
we show that a inhomogeneous almost symmetric submanifold is
of cohomogeneity one, and we construct examples of compact
full irreducible almost symmetric submanifolds of Euclidean space
with cohomogeneity one. A fuller investigation will be the
subject of future work.

The simplest example of compact full irreducible almost symmetric
submanifold of Euclidean space of cohomogeneity one is as follows.

\begin{eg}
Let $M$  be an ellipsoid of
$SO(n)$-revolution in $\R^n\times\R\cong\R^{n+1}$, which is not a sphere.  
Then $M$ is foliated by $n-1$-spheres, each such $n-1$-sphere $\Sigma$ 
spans an affine subspace $\langle\Sigma\rangle$ parallel 
to $\R^n\times\{0\}$, and 
the almost symmetry at a point in $\Sigma$,
different from the two vertices, is the reflection 
on the affine $2$-plane $\langle\Sigma\rangle^\perp$. 
Note that~$M$ contains singular orbits of the isometry group $SO(n)$,
namely, the vertices. Of course, at the vertices almost symmetries 
exist and are not unique. 
If we delete the vertices from $M$, we get an 
example of incomplete almost symmetric submanifold. 
\end{eg}

A more general construction is given by the following
family of examples.  

\begin{eg}
Contemplate an irreducible $s$-representation of a compact
connected Lie group $K$ on $\mathbb{R}^n$, and suppose that $K\cdot w$
is a symmetric orbit. Consider also an action of
$G = SO(k_1)\times SO(k_2)\times K$ on 
$\mathbb{R}^{k_1}\oplus \mathbb{R}^{k_2}\oplus \mathbb{R}^n$ in the 
standard way, and fix
$v=e_1 +  e_2 +w$, where 
where $e_i\in \mathbb{R}^{k_i}$ is of unit length 
for $i=1$, $2$. 
Finally, define the curve 
$\gamma (t) = a\cos(t) e_1 + b\sin(t) e_2 + (\cos(2t) +2) w$,
with arbitrary $a$, $b>0$.
Then the image of $\gamma (t)$ is invariant under the reflection 
on $e_i^\perp$ for $i=1$, $2$
(in the space $\mathbb{R} e_1 \oplus \mathbb{R} e_2 \oplus \mathbb{R}w$), 
and it is easy to see that
 $G\cdot \gamma(\mathbb{R})$ is an extrinsic almost symmetric submanifold 
diffeomorphic to $S^{k_1+k_2-1}\times K\cdot w$.
\end{eg}

Finally, we discuss why the cohomogeneity one property of the above examples
is not unexpected. Indeed below we shall consider a more general situation 
of an \emph{intrisically almost symmetric space} $M$, namely,
for each $p\in M$ there exists an (intrinsic) isometry of $M$ of order~$2$
fixing $p$ and having a fixed point set of dimension $1$ in $T_pM$. 

We start with a complete Riemannian manifold $M$, and let $K$ be a (possibly 
disconnected) closed subgroup of its isometry group. 
Then $K$ acts properly on $M$, and the $K$-orbits in $M$ are 
properly embedded submanifolds. Assume that $K$ does not act 
transitively on~$M$, and fix a connected component $N$ of 
a $K$-orbit, so that $N=K^0p$ for some $p\in M$. Let $\mathcal T$ 
be an open tubular neighborhood of $N$ in $M$. We can and will
assume that~$\mathcal T$
has radius~$\epsilon>0$ smaller than half of the distance from $N$ to 
any other connected component of $Kp$. Denote by $\pi:\mathcal T\to N$
the map that assigns a point in $\mathcal T$ to its closest point in $N$.

\begin{lem}
The isotropy groups $K_z\subset K_{\pi(z)}$ for all $z\in\mathcal T$. 
\end{lem}

\Pf For $z\in\mathcal T$, there is a unique shortest geodesic segment
$\gamma$ from $z$ to $\pi(z)$. Let $k\in K_z$. Then $k\circ\gamma$ 
is a geodesic segment of the same length from $z$ to $Kp=K\cdot N$. 
By our choice of $\epsilon$, we must have $k\cdot N=N$. 
Due to the uniqueness above, $k\circ\gamma=\gamma$.
In particular, $k\pi(z)=\pi(z)$, as wished. \EPf

\begin{rmk}\label{several}
If $M$ as above is in addition almost symmetric, and we take 
$K$ to contain the almost symmetries at all points, then we see that  
the almost symmetry $\sigma_z$ at $z\in \mathcal T\setminus N$ is unique. 
In fact it must fix the geodesic $\gamma$ through $z=\gamma(0)$ and 
$q=\pi(z)=\gamma(1)$. It also 
follows that 
$d\sigma_z(\gamma'(1))=\gamma'(1)$, and that 
$\sigma_z$ is an almost symmetry at~$q$. This construction indeed 
shows that for all $q\in N$ and all $\xi\in\nu_qN$, there exists an 
almost symmetry of $M$ at $q$ fixing $\xi$. It follows that 
almost symmetries along points of $N$ are not unique if the codimension
of $N$ in $M$ is bigger than one.   
\end{rmk}

We have come to the proof of of the next main result of the paper.

\textit{Proof of Theorem~\ref{thm:inhomog}.}
Let $p\in\Omega$. Then $Kp$ is a principal orbit.
It follows that the slice representation at~$p$ is trivial. If $\sigma$ 
is an almost symmetry at $p$, then $\sigma\in K_p$, so $d\sigma_p$ 
is the identity on $\nu_p(Kp)\subset T_pM$. But the fixed point
set of $d\sigma_p$ in $T_pM$ is one-dimensional, therefore
we must have $\dim\nu_p(Kp)=1$. Since $Kp$ is a principal orbit,
this proves~(a). Also the fix point set of $d\sigma_p$ 
must be $\nu_p(Kp)$, so this completly determines $\sigma$ 
and proves~(b).

Let $N$ be a $K^0$-orbit of the codimension $k\geq2$. Owing to 
Remark~\ref{several}, for all $p\in N$ and $\xi\in\nu_pN$, there 
exists a unique almost symmetry $\sigma_\xi$ at $p$ such that 
$d(\sigma_\xi)_p(\xi)=\xi$; then we necessarily have
$d(\sigma_\xi)_p|_{T_pN}=-\mathrm{Id}_{T_pN}$. 
If $\xi$, $\eta\in\nu_pN$ are 
linearly independent, then $\sigma_\xi\circ\sigma_\eta \in K_p$,
and $d(\sigma_\xi\circ\sigma_\eta)_p$ is a rotation in the $2$
-plane $\Pi$ spanned by $\xi$, $\eta$, and the identity on the orthogonal
complement $\Pi^\perp$ in $T_pM$. If we take $\xi$ and $\eta$ forming 
an angle incommensurate with $\pi$, 
the cyclic group generated by $\sigma_\xi\circ\sigma_\eta$ acting on $\Pi$
is dense in $SO(\Pi)\cong SO(2)$.
Let $H$ denote the closure of the 
subgroup of $K_p$ generated by products $\sigma_\xi\circ\sigma_\eta$,
where  $\xi$, $\eta\in\nu_pN$ are linearly independent.
It is clear that $H$ is a connected subgroup of $K$. 
Since $\Pi$ can be any $2$-plane in $\nu_pN$, the above discussion shows
$H$ acts on $T_pM=T_pN\oplus\nu_pN$ as $\mathrm{Id}\times SO(\nu_pN)$. 
Since the isotropy representation is faithful,
the Lie algebra of $H$ is isomorphic to $\mathfrak{so}_k$.
Note that $H$ is a normal subgroup of $K^0$ (and more generally of the
$K$-stabilizer of $N$), since it precisely consists of those
elements that act trivially on $N$. 
This proves~(c) and~(d), because $N$ is the connected component through~$p$
of the fixed point set of $H$. 
Finally, part~(e) is immediate from the existence of 
the almost symmetry $\sigma_\xi$ of $M$ fixing any given 
$\xi\in\nu_q(Kq)$ and restricting to minus identity along
$T_q(Kq)$, where $q\in M$ is arbitrary. \EPf

\medskip

We will use the following lemmata in the proof
of Theorem~\ref{thm:inhomog-classif}, the proof
of the first one being strightforward. 

\begin{lem}\label{10}
Let $G$, $H$, $K$ be Lie groups such that 
$K$ is a compact subgroup of $G$ and $H$ is compact,
and let $\rho:K\times H\to O(V)$ 
be a representation such that $\rho(K)$ is trivial. Then 
the map 
\[ \Phi: (G\times H)_{K\times H}V \to G/K \times V,\ \Phi[(g,h),v]=(gK,hv) \]
between homogeneous vector bundles over~$G/K$
defines a $G$-equivariant diffeomorphism.
\end{lem}

\begin{lem}\label{20}(\cite{AA,GWZ},\cite[Ch.~IV, Thms.~8.1 and~8.2]{Br})
A cohomogeneity $1$ Riemannian manifold is determined by its 
isotropy groups. 
\end{lem}

\textit{Proof of Theorem~\ref{thm:inhomog-classif}.}
Let $M$ be a simply-connected complete almost symmetric space. Denote by $K$
its group of isometries. Suppose that $M$ is inhomogeneous.
Due to Theorem~\ref{thm:inhomog}, $K$ acts on $M$ with cohomogeneity~$1$
and has principal orbits that are symmetric spaces. It follows
that the principal orbit $K/H$ is the product of certain
irreducible symmetric spaces,
and the regular set $M_{reg}$
is a multiply warped product $(0,1)\times_fK/H$ where the metric
on $N$ is scaled independently on each factor.

We also know that there are no exceptional orbits, and that 
$K$ and $K^0$ have the same orbits~\cite[pp.~41-42]{GWZ}.
For the sake of simplicity, in the sequel we work with $K^0$ and drop the
superscript from the notation. 

According to~\cite{Mo}, the orbit space $M/K$ is one of the
following:
\begin{enumerate}
\item[(a)] $(0,1)$;
\item[(b)] $S^1$ ;
\item[(c)] $[0,1)$;
\item[(d)] $[0,1]$.
\end{enumerate}

In case~(a) there are no singular orbits and $M=(0,1)\times_f K/H$.
Case~(b) cannot occur due to the assumption that $M$ is simply-connected. 

The following remark will be useful.
Suppose $q\in M$ is a singular point, and denote by $k\geq2$ the 
codimension of the singular orbit $N=Kq$ in $M$. Then $N$
is a symmetric space and
$K=L\times K_1$, where $K_1\cong SO(k)$.
Also, the isotropy group 
$K_q=L_q\times K_1$, and a tubular neighborhood of
$N$ in $M$ is diffeomorphic to
\[K\times_{K_q}\R^k= L/L_q \times\R^k = N \times \R^k, \]
namely, a product, by Lemma~\ref{10}. 

In case (c) there is only one singular orbit, so we deduce that $M$
is a Riemannian product $L/H \times \R^k$, where $L/H$ is a
simply-connected symmetric space and 
$\R^k$ carries a rotationally invariant metric. 

In case (d) there are two singular orbits, and there are two subcases.
In the first one, $K=K_1\times K_2\times L$, where $K_i\cong SO(k_i)$,
$k_i\geq2$ and $i=1$, $2$. The principal orbits are symmetric
spaces $K_1/H_1\times K_2/H_2\times L/H$, and the singular orbits
are $K_1/H_1\times L/H$ and $K_2/H_2\times L/H$, where $H_i\cong SO(k_i-1)$
for $i=1$, $2$.  

The action of $K_1\times K_2\cong SO(k_1)\times SO(k_2)$ on 
$S^{k_1+k_2-1}\subset\R^{k_1}\times\R^{k_2}$ has cohomogeneity~$1$
with singular isotropy groups $H_1\times K_2$ and $K_1\times H_2$, 
and principal isotropy group $H_1\times H_2$. 
Let now $(K_1\times K_2)\times L$ act on $S^{k_1+k_2-1}\times L/H$
componentwise. This is an action of cohomogeneity~$1$ with
isotropy groups as above and hence $M$ is diffeomorphic
to $S^{k_1+k_2-1}\times L/H$ by Lemma~\ref{10}.
The principal orbits are $S^{k_1}\times S^{k_2}\times L/H$.
The admissible metrics scale the two spheres and each irreducible
factor of $L/H$ independently as a function of $t$, the arc-length parameter
along a normal geodesic.

The second subcase is $K=K_1\times L$, where $K_1\cong SO(k_1)$
for $k_1\geq2$. The principal orbits are symmetric spaces
$K_1/H_1\times L/H$, 
and the two singular orbits are $K_1/H_1\times L/H$, where $H_1\cong SO(k_1-1)$.
Here $M$ is diffeomorphic to $S^{k_1}\times L/H$. The
principal orbits are $S^{k_1-1}\times L/H$, and  
the admissible metrics scale the sphere and each irreducible
factor of $L/H$ independently as a function of $t$, the arc-length parameter
along the normal geodesic.

Hence in either subcase $M$ is as stated in the theorem. \EPf

\providecommand{\bysame}{\leavevmode\hbox to3em{\hrulefill}\thinspace}
\providecommand{\MR}{\relax\ifhmode\unskip\space\fi MR }
\providecommand{\MRhref}[2]{%
  \href{http://www.ams.org/mathscinet-getitem?mr=#1}{#2}
}
\providecommand{\href}[2]{#2}


\end{document}